\documentclass[12pt]{article}
\usepackage{epsfig,psfrag,amsmath,amssymb,latexsym}
\usepackage{amscd}
\usepackage{color}
\usepackage{amsfonts}
\usepackage{graphicx}
\usepackage{wasysym}
\usepackage{mathrsfs}
\usepackage{verbatim}
\usepackage{hyperref}
\usepackage{cancel}
\usepackage[utf8]{inputenc}
\numberwithin{equation}{section}
\newtheorem{thm}{Theorem}[section]

\newtheorem{theorem}[thm]{Theorem}
\newtheorem{fact}[thm]{Fact}

\newtheorem{definition}[thm]{Definition}
\newtheorem{lemma}[thm]{Lemma}
\newtheorem{remark}[thm]{Remark}

\newenvironment{proof}[1][Proof]{\textbf{#1.} }{\ \rule{0.5em}{0.5em}}

\newcommand{\htwo}{H^{2|2}}
\newcommand{\hthree}{H^{3|2}}
\newcommand{\hthreeplus}{H^{3|2}_+}

\newcommand{\Id}{\operatorname{Id}}

\newcommand{\diag}{\operatorname{diag}}

\newcommand{\rhs}{{\mathrm{rhs}}}
\newcommand{\wneq}{W^{\neq}}
\newcommand{\body}{\operatorname{body}}

\newcommand{\ig}{\operatorname{IG}}

\newcommand{\sk}[1]{\left\langle{#1}\right\rangle}

\newcommand{\cA}{\mathcal{A}}
\newcommand{\cD}{\mathcal{D}}

\newcommand{\B}{\mathcal{B}}
\newcommand{\E}{{\mathbb E}}

\newcommand{\R}{{\mathbb R}}  
  
\newcommand{\N}{{\mathbb N}}  
\newcommand{\Z}{{\mathbb Z}}  
  
\newcommand{\F}{{\mathcal{F}}}

\newcommand{\cnull}{{c_2}}
\newcommand{\cdrei}{{c_1}}

\begin{document}
\thispagestyle{empty}

\begin{center}
  {\LARGE Restrictions of some reinforced processes to subgraphs}\\[3mm]
{\large Margherita Disertori\footnote{Institute for Applied Mathematics
\& Hausdorff Center for Mathematics, 
University of Bonn, \\
Endenicher Allee 60,
D-53115 Bonn, Germany.
E-mail: disertori@iam.uni-bonn.de}
\hspace{1cm} 
Franz Merkl \footnote{Mathematisches Institut, Ludwig-Maximilians-Universit{\"a}t  M{\"u}nchen,
Theresienstr.\ 39,
D-80333 Munich,
Germany.
E-mail: merkl@math.lmu.de
}
\hspace{1cm} 
Silke W.W.\ Rolles\footnote{
Department of Mathematics, CIT, 
Technische Universit{\"{a}}t M{\"{u}}nchen,
Boltzmannstr.\ 3,
D-85748 Garching bei M{\"{u}}nchen,
Germany.
E-mail: srolles@cit.tum.de}
\\[3mm]
{\small \today}}\\[3mm]
\end{center}

\begin{abstract}
  We prove that the restriction of the vertex-reinforced jump process
  to a subset of the vertex set is a mixture of vertex-reinforced jump
  processes. A similar statement holds for the non-linear hyperbolic
  supersymmetric sigma model. This is then applied to vertex-reinforced
  jump processes on subdivided versions of graphs of bounded degree,
  where every edge is replaced by a finite sequence of edges. We prove
  that discrete-time processes associated to suitable corresponding
  restrictions are mixtures of positive recurrent Markov chains.
  We also deduce a similar statement for edge-reinforced random walks. 
  \footnote{MSC2020 subject classifications: Primary 60K35, 60K37; secondary 60G60.}
\end{abstract}

\section{Models and Results}
\label{sec:Models-and-Results}

\subsection{Motivation}
\label{sec:Motivation}
One of the biggest open problems concerning vertex-reinforced jump processes, vrjp
for short, is to decide whether the discrete-time process
associated to vrjp on $\Z^2$ is a mixture of \emph{positive} recurrent
Markov chains for all constant initial weights. For small weights, this
problem was solved by Sabot and Tarr\`es \cite{sabot-tarres2012} and with
a completely different technique by Angel, Crawford, and Kozma 
\cite{angel-crawford-kozma}. 
The corresponding statement for recurrence rather than
positive recurrence has been proven for all constant initial weights 
by Sabot in \cite{sabot-polynomial-decay2021}. Solving the above mentioned open
problem seems currently out of reach. 
A possible approach might be the development of
a renormalization group technique for vrjp on $\Z^2$, restricting it to
smaller and smaller sublattices. We show in this paper that restriction of
vrjp to subsets of a finite vertex set is a mixture of vrjp with
random weights, which gives rise to a kind of renormalization flow
on the distributions of these random weights.
As a case study we analyze this flow on subdivided graphs, showing that
it drives the effective random weights towards smaller and smaller
values in a stochastic sense and thus more and more into the positive
recurrent regime.

\subsection{Reinforced processes}
\label{subsec:Reinforced-processes}

Let $G=(\Lambda,E)$ be an undirected locally finite connected graph
with vertex set $\Lambda$ endowed with edge weights $W_e=W_{ij}>0$, $e=\{i,j\}\in E$. 
The vertex-reinforced jump process (vrjp) on $G$ is a continuous-time jump process
$(Y_t)_{t\ge 0}$ taking values in $\Lambda$. Conditioned on $(Y_s)_{s\le t}$
and $Y_t=i\in\Lambda$,
it jumps to a neighboring vertex $j\in\Lambda$ at the rate
$W_{ij}L_j(t)$ with $L_j(t)=1+\int_0^t1_{\{Y_s=j\}}\, ds$ being the local time at
$j$ with an offset of 1. Vrjp was invented by Werner in 2000.
Sabot and Tarr\`es \cite{sabot-tarres2012} introduced the vrjp in exchangeable time
scale $(Z_t:=Y_{D^{-1}(t)})_{t\ge 0}$ with the time change $D(t)=\sum_{i\in\Lambda}(L_i(t)^2-1)$.
In this time scale, vrjp is a mixture of reversible Markov jump processes; see
\cite[Theorem 2]{sabot-tarres2012} and \cite[Theorem 1]{sabot-zeng15}.

We remark that the vrjp on infinite graphs might make infinitely many jumps
in finite time, which means explosion in finite time, if the weights $W_e$ increase
fast enough far out. However, on finite graphs, this does not occur almost
surely. 

For technical reasons, we encode a continuous-time jump process on $G$ by
two sequences of random
variables $(X_n)_{n\in\N_0}$ and $(T_n)_{n\in\N_0}$. The random variable $X_n$
takes values in $\Lambda$; it encodes the $n$-th position visited.
The event $\{X_n=X_{n+1}\}$ may occur with positive probability. If such an
event occurs, we say that the process has a self-loop. The random variable
$T_n$ takes positive real values; it encodes the waiting time for the jump
from $X_n$ to $X_{n+1}$.
The connection of this description to a continuous-time $\Lambda$-valued jump process
$(Y_t)_{t\ge 0}$ can be described on the event $\{\sum_{n=0}^\infty T_n=\infty\}$
by $Y_t=X_n$ for $\sum_{l=0}^{n-1}T_l\le t<\sum_{l=0}^n T_l$. Note that in the
representation $(Y_t)_{t\ge 0}$ the information on self-loops is lost.
In the case of explosion in finite time, the sum $\sum_{n=0}^\infty T_n$ is finite.
In this case, $Y_t$ is only defined for $t<\sum_{n=0}^\infty T_n$, but 
the description in terms of $X_n,T_n$ still exists for all $n$.
This is why we do not need any assumptions
on the weights $W$ that avoid explosions in finite time. 

Our first result concerns the vrjp restricted to a subset $J\subseteq\Lambda$.
We define the restriction to a subset for a general continuous-time jump process. 

\begin{definition}[Removal of self-loops and restriction to a subset: process]
  \label{def:restriction-to-a-subset}
\mbox{}\\
  Let $(X_n,T_n)_{n\in\N_0}$ be a continuous-time jump process on $G$.
  Recursively, we take $\sigma_0:=0$ and for $n\in\N$, on the event
  $\{\sigma_{n-1}<\infty\}$, we define 
  $\sigma_n:=\inf\{l>\sigma_{n-1}:X_l\neq X_{\sigma_{n-1}}\}$ to be the index
  of the next jump to a different location. 
  The process with self-loops removed is defined by 
  $(X^{\neq},T^{\neq})=(X^{\neq}_n,T^{\neq}_n)_{n\in\N_0}:=(X_{\sigma_n},\sum_{l=\sigma_n}^{\sigma_{n+1}-1}T_l)_{n\in\N_0}$
  on the event $\{\sigma_n<\infty\text{ for all }n\in\N\}$. 

  Let $\emptyset\neq J\subseteq\Lambda$. We set recursively $\tau_0=0$, 
  $\tau_n=\inf\{ l>\tau_{n-1}: X_l\in J\}$
  for $n\in\N$. In other words, on the event $\{X_0\in J\}$, 
  $\tau_n$ denotes the number of jumps up to the $n$-th return to $J$. 
  The restriction of the process to $J$ is defined
  on the event $\{\tau_n<\infty\text{ for all }n\in\N\}$ by
  $(X^J,T^J)=(X_n^J,T_n^J)_{n\in\N_0}:=(X_{\tau_n},T_{\tau_n})_{n\in\N_0}$.

  The notation $(X^{J\neq},T^{J\neq})$ means that both operations have
  been applied to the process $(X,T)$, first the restriction to $J$ and
  then self-loop removal.
\end{definition}

One may visualize the restriction as editing a film of the continuous
time representation of the jump process, cutting out all parts of the film
where the jumping particle is not in~$J$, but the cut locations in
the edited film remain visible as self-loops. Removal of these self-loops
in this edited film means that the corresponding cut locations become invisible. 

Note that the definitions of $X^{\neq}$ and $X^J$ do not use the
  $T$-components of the process. In particular, the definitions of $X^{\neq}$, $X^J$,
  and $X^{J\neq}$ make also sense if one starts with a discrete-time process
  $(X_n)_{n\in\N_0}$ only. 

The mixing measure representing vrjp in exchangeable time scale as a mixture of
Markov jump processes has been described in terms of a random field
$\beta_\Lambda=(\beta_i)_{i\in\Lambda}$
in \cite[Proposition 2]{sabot-tarres-zeng2017} for finite graphs 
and in \cite[Theorem 1]{sabot-zeng15} for infinite graphs. Because the law
$\nu^W_\Lambda$ of this random field $\beta_\Lambda$ appears in the present paper
in an additional role, we review it first for a finite set
$\Lambda$ including a pinning point $\rho\in\Lambda$. Take a
symmetric matrix $W=(W_{ij})_{i,j\in\Lambda}\in[0,\infty)^{\Lambda\times\Lambda}$
of weights.
Note that $W$ may have positive diagonal entries.
For $\beta\in\R^\Lambda$, define
\begin{align}
  \label{eq:def-H-beta}
  H_{\Lambda,\beta}^W=H_\beta^W=H_\beta:=2\diag(\beta)-W, 
\end{align}
where $\diag(\beta)$ denotes the diagonal matrix with diagonal entries
given by $\beta_i$, $i\in\Lambda$. 
Let $1\in\R^\Lambda$ denote the column vector having all entries equal
to 1, which implies that the Euclidean inner product $\sk{1,H_\beta 1}$ is the
sum of all entries
of the matrix $H_\beta$. The law of $\beta_\Lambda$ equals the probability measure
\begin{align}
  \label{eq:def-nu}
  \nu^W_\Lambda(d\beta)
  := \left(\frac{2}{\pi}\right)^{\frac{|\Lambda|}{2}}1_{\{ H_\beta>0 \}}
  \frac{e^{-\frac12\sk{1,H_\beta 1}}}{\sqrt{\det H_\beta}}\, d\beta
\end{align}
on $\R^\Lambda$, where the notation $H_\beta>0$ means that the matrix $H_\beta$ is
positive definite. The probability measure $\nu^W_\Lambda$ was 
introduced for $W_{ii}=0$ in 
\cite[Definition 1]{sabot-tarres-zeng2017} and generalized for $W_{ii}\ge 0$
in \cite[Section 5.1]{sabot-zeng15}; see also \cite[Section 4]{letac-wesolowski2020}.
\cite[Proposition 1]{sabot-zeng15} extends the definition of $\nu^W_\Lambda$
to infinite graphs.

We use the following notation. For a vector $v\in\R^K$, a matrix
$A\in\R^{K\times L}$, and subsets $I\subseteq K$ and $J\subseteq L$ of the index
sets, we denote by $v_I$ the restriction of $v$
to $I$ and by $A_{IJ}$ the restriction of $A$ to $I\times J$.
Let $P_\rho^{W,\Lambda}$ denote the law of the vrjp in exchangeable time scale
on $\Lambda$ starting in $\rho$ with weights $W$.

\begin{theorem}[Restriction of vrjp as a mixture of vrjps]
  \label{thm:mixture-of-vrjps}
  \mbox{}\\
  Assume that the graph $G$ is finite without self-loops and partition its vertex set $\Lambda=I\cup J$,
  $I\cap J=\emptyset$, with $|J|\ge 2$. 
  Consider the vrjp $(X,T)$ in exchangeable time scale on $\Lambda$
  starting at $\rho\in J$ with weights $W$. 
  The restrictions $(X^J,T^J)$ and $(X^{J\neq},T^{J\neq})$ to $J$ without or with
  self-loops removed are
  mixtures of vrjps in exchangeable time scale on $J$ with random weights
    \begin{align}
    \label{def:W-J}
    W^J(\beta_I)=&(W^J_{ij}(\beta_I))_{i,j\in J}:=W_{JJ}+W_{JI}([H_\beta]_{II})^{-1}W_{IJ}
                    \text{ and }\\
      \label{eq:def-W-J-neq}
    W^{J\neq}(\beta_I):=&\left(W^J_{ij}(\beta_I)1_{\{ i\neq j\}}\right)_{i,j\in J},
  \end{align}
  respectively. They depend on a random
  vector $\beta_I\in\R^I$, where $\beta_{I\cup\{\rho\}}\in\R^{I\cup\{\rho\}}$ is distributed
  according to $\nu_{I\cup\{\rho\}}^{\widehat W}$ with $\widehat W\in\R^{(I\cup\{\rho\})\times(I\cup\{\rho\})}$
  obtained by restricting the parameters $W$ to $I$ and wiring all points
  in $J$ at $\rho$:
  \begin{align}
    \label{eq:def-W}
    \widehat W_{ij}=\widehat W_{ji}=
    \begin{cases}
      W_{ij} &\text{ for }i,j\in I,\\
      \sum_{k\in J}W_{ik} &\text{ for }i\in I, j=\rho,\\
      0 &\text{ for }i=j=\rho.
    \end{cases}
  \end{align}
  In the case of $(X^{J\neq},T^{J\neq})$ this means the
  following for any event $A\subseteq J^{\N_0}\times\R_+^{\N_0}$. 
  \begin{align}
    P_\rho^{W,\Lambda}((X^{J\neq},T^{J\neq})\in A)
    =&\int_{\R^{I\cup\{\rho\}}}
       P_\rho^{W^{J\neq}(\beta_I),J}((X,T)\in A)\,\nu_{I\cup\{\rho\}}^{\widehat W}(d\beta_{I\cup\{\rho\}})\nonumber\\
    =&\int_{\R^\Lambda}
       P_\rho^{W^{J\neq}(\beta_I),J}((X,T)\in A)\,\nu_\Lambda^W(d\beta_\Lambda).
       \label{eq:mixture-of-vrjps}
  \end{align}
  The analogous statement holds for $(X^J,T^J)$. 
\end{theorem}
Note that on the l.h.s.\ in \eqref{eq:mixture-of-vrjps} the process $(X^{J\neq},T^{J\neq})$
is built from the canonical process $(X,T)$
on $\Lambda^{\N_0}\times\R_+^{\N_0}$, while on the r.h.s.\ $(X,T)$
means the canonical process on $J^{\N_0}\times\R_+^{\N_0}$.

 An explicit formula for the probability density of $\beta_I$ can be
 found in \cite[Lemma 4 combined with Lemma 5(i)]{sabot-zeng15}; however
 we do not need it here. 

\paragraph{Comparison with the restriction property observed by Davis and Volkov.}
In the special case of $J$ being a set of consecutive integers
on a one-dimensional integer interval this restriction property
has already been observed by Davis and Volkov in
\cite[Section 3]{davis-volkov1}. In this special case, $W^{J\neq}$
is deterministic and equals
the restriction of $W$ to $J\times J$. Hence, in this special case,
$(X^{J\neq},T^{J\neq})$ is again a vrjp, not only a mixture of vrjps.
The analogous property holds on a tree. 

\paragraph{Subdivisions.}
For $r\in\N_0$, we define the $2^r$-subdivision $G_r=(\Lambda_r,E_r)$ of
the undirected graph $G=(\Lambda,E)$ to be obtained by replacing every edge in
$G$ by a series
of $2^r$ edges; see Figure \ref{fig:subdivision} for an illustration.
It will be convenient to have $\Lambda_0=\Lambda$ and
$\Lambda_l\subseteq\Lambda_r$ for $l\le r$; see also Definition \ref{def:subdivided-graph}
below. 

\begin{figure}
  \centering
\includegraphics[width=0.8\textwidth]{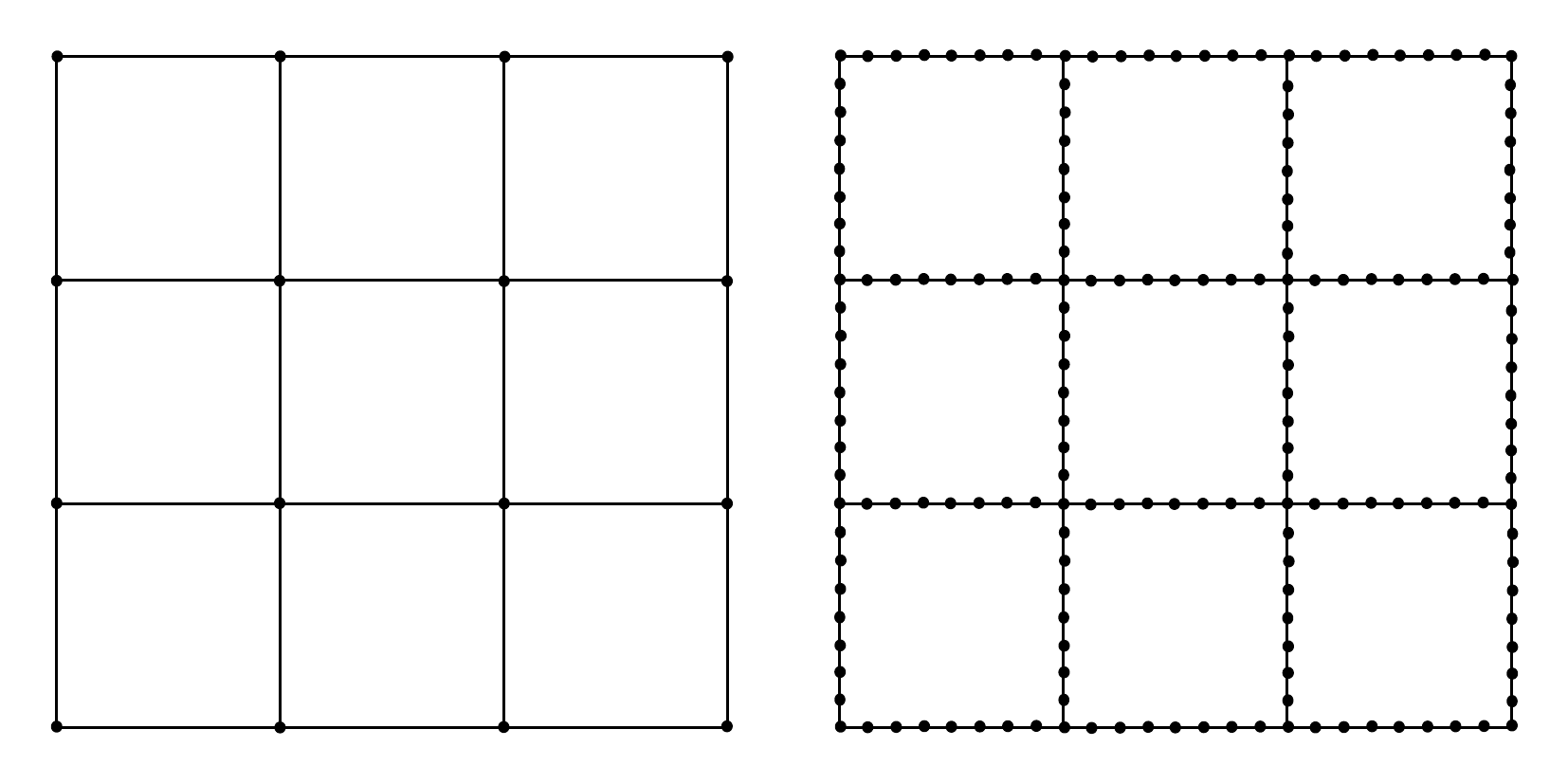}
\caption{A part of $\Z^2$ on the left and its 8-subdivided version on the right.}
  \label{fig:subdivision}
\end{figure}

The next result considers vrjp on a subdivided graph $G_r$ with random weights
$W$. If the graph has degree bounded by $d$, the following
result of Sabot and Tarr\`es allows us to deduce recurrence of the restriction
to $G_l$ with $l<r$, 
provided $r-l$ is large enough depending on $\E[W_e^\alpha]$, $d$, and $\alpha$.

\begin{fact}[{\cite[Corollary 3]{sabot-tarres2012}; see also
    \cite[Theorem 20]{angel-crawford-kozma}}]
  \label{fact:sabot-tarres-recurrence}
  \mbox{}\\
  Let $d\in\N$ and $\alpha\in(0,\frac14]$. Then, there is $\cdrei=\cdrei(d,\alpha)>0$ 
  such that for all connected undirected graphs $G=(\Lambda,E)$ with vertex degree
  bounded by $d$, all independent random weights $W=(W_e)_{e\in E}$
  (not necessarily identically distributed) with $\E[W_e^\alpha]\le\cdrei$ for
  all $e\in E$, and all starting points $\rho\in \Lambda$, 
  the discrete-time process associated to vrjp with random weights $W$
  starting in $\rho$ is a
  mixture of positive recurrent Markov chains. 
\end{fact}

\begin{theorem}[Vrjp on subdivided graphs]
  \label{thm:vrjp-on-G-r-pos-rec}
  Let $G=(\Lambda,E)$ be a connected undirected graph without self-loops
  and take $l,r\in\N_0$ with $l\le r$.
  Let $(X,T)$ be vrjp in exchangeable time scale on the subdivided graph $G_r$ with
  starting point $\rho\in\Lambda$ and random weights $W_e>0$, $e\in E_r$, with
  respect to some probability measure with corresponding expectation $\E$. Its restriction
  $(X^{\Lambda_l\neq},T^{\Lambda_l\neq})$ to $\Lambda_l$
  with self-loops removed is again a mixture of vrjps on $G_l$ with random
  weights denoted by $W^{(l)}=(W_e^{(l)})_{e\in E_l}$. If the family $W=(W_e)_{e\in E_r}$ is
  independent or i.i.d., then so is the family $W^{(l)}$. Given $W$, the conditional
  law of $W^{(l)}$ can be described in terms of the restriction
  $\beta_{\Lambda_r\setminus\Lambda_l}$ of a $\nu_{\Lambda_r}^W$-distributed
  random field $\beta$. 
  For any finite subgraph $\tilde G$ of $G$ and the corresponding subdivision
  $\tilde G_l=(\tilde\Lambda_l,\tilde E_l)$, the restriction of $W^{(l)}$ to
  $\tilde E_l$ equals $W^{\tilde\Lambda_l}$ given in \eqref{def:W-J} with $J=\tilde\Lambda_l$
  and $I=\tilde\Lambda_r\setminus\tilde\Lambda_l$, and
  fulfills the recursion equations described in 
  Lemma \ref{le:aux-W-l-1}, below. In particular, one has 
  \begin{align}
    \label{eq:mixture-of-vrjps-lambda-r}
    &P_\rho^{W,\Lambda_r}((X^{\Lambda_l\neq},T^{\Lambda_l\neq})\in\cdot)=\int_{\R^{\Lambda_r}}
      P_\rho^{W^{(l)\neq}(\beta),\Lambda_l}((X,T)\in\cdot)\,
      \nu_{\Lambda_r}^W(d\beta).
  \end{align}
  Assume that the vertex degree of $G$ is bounded by $d$. Moreover, assume
  that the weights $W_e$, $e\in E_r$, are i.i.d.\ and satisfy
  $\E[W_e^\alpha]\le\cdrei 2^{\alpha(r-l)}$ for some $\alpha\in (0,\tfrac14]$
  with the 
  constant $\cdrei(d,\alpha)$ from Fact \ref{fact:sabot-tarres-recurrence}. 
  Then, the process $X^{\Lambda_l\neq}$ is a mixture of positive recurrent
  reversible Markov chains.
\end{theorem}

One could weaken the assumption of the recurrence statement by making only
an independence assumption rather than an i.i.d.\ assumption. To increase
readability of the proof, we only treat the stronger assumption.

\paragraph{Consequences for linearly edge-reinforced random walk.}
Linearly edge-reinfor\-ced random walk (errw) $(X_n)_{n\in\N_0}$
on $G$ starting at $\rho$ with constant initial weights $a>0$
is defined as follows. Let $X_0=\rho$ and let $w_0(e):=a$, $e\in E$, be the initial edge
weights. In each time step, the random walker jumps to a neighboring vertex
with probability proportional to the weight of the traversed edge.
Each time an edge is traversed, its weight is increased by 1. 
More formally, conditioned on $(X_m)_{m\le n}$ and
$X_n=i\in\Lambda$, the conditional probability of $X_{n+1}=j\in\Lambda$
is non-zero only if $\{i,j\}\in E$. It equals 
$w_n(\{i,j\})/\sum_{k\in\Lambda:\{i,k\}\in E}w_n(\{i,k\})$,
where $w_n(e)=a+\sum_{m=0}^{n-1}1_{\{\{ X_m,X_{m+1}\}=e\}}$ denotes the
weight of the edge $e$ at time $n$. 
Errw was introduced by Diaconis in 1986 in \cite{Coppersmith-Diaconis-unpublished},
see \cite{Diaconis88}. For more history on this process, see
\cite{merkl-rolles-festschrift-2006}. 
It was shown in \cite[Theorem 1]{sabot-tarres2012}, that errw
is a mixture of the discrete-time process associated to vrjp with i.i.d.\
Gamma($a$,1)-distributed weights $W_e$, $e\in E$. 
The following result shows that under rather general conditions, 
errw on a subdivided graph $G_r$ with constant initial weights, appropriately
restricted, becomes
a mixture of \emph{positive recurrent} Markov chains as soon as $r$ is
large enough. More precisely, we prove the following statement.

\begin{theorem}[Errw on subdivided graphs]
  \label{thm:Errw-on-subdivided-graphs}
  Let $G=(\Lambda,E)$ be a connected undirected graph without
  self-loops and with vertex degree
  bounded by $d$. For all $r\in\N_0$ consider the subdivided graph $G_r$. 
  Let $X$ be errw on $\Lambda_r$  with starting point $\rho\in\Lambda$
  and with constant initial weights $a>0$. 
  Assume that $\alpha\in\left(0,\tfrac14\right]$, $r\in\N_0$, and 
  $l\in\{0,\ldots,r\}$ satisfy
  $\Gamma(a+\alpha)/\Gamma(a)\le\cdrei(d,\alpha)2^{\alpha(r-l)}$
  with $\cdrei$ as in Fact \ref{fact:sabot-tarres-recurrence}. Then, 
  the restriction $X^{\Lambda_l\neq}$ of errw to $\Lambda_l$ with self-loops
  removed is a mixture of
  positive recurrent Markov chains.
\end{theorem}

In this paper, we treat $2^r$-subdivisions with powers of $2$ only rather than
arbitrary $k$-subdivisions for $k\in\N$. This is just for notational
simplicity. General $k$-subdivisions could be treated by the same method, 
applying the recursive restriction described in Lemma \ref{le:aux-W-l-1}, below,
not to all edges simultaneously in every recursion step, but
only to some of the edges. 

\paragraph{Comparison with previous work on errw.}
\cite[Theorem 1.1]{merkl-rolles-2d} provides a variant of Theorem
\ref{thm:Errw-on-subdivided-graphs} in 
the special case of the graph $\Z^2$ with nearest-neighbor edges, 
showing only recurrence rather than a mixture of positive recurrent Markov chains.
Note that at that time, the
relation between errw and vrjp, which is an essential tool in the proof
of  Theorem \ref{thm:Errw-on-subdivided-graphs}, was not known.
The present paper gives a heuristic explanation why subdivisions
make errw and vrjp more recurrent: the reason is that taking subdivisions
and then restricting to the original graph decreases the effective weights
in a stochastic sense. This is made precise in the next lemma.
By a monotonicity result of Poudevigne \cite[Theorem 1]{poudevigne22} decreasing
the weights of vrjp increases the probability of vrjp being recurrent.

\begin{theorem}[Decay of the effective weights by restriction]
  \label{thm:from-r-to-l}
  Let the graph $G$ be finite and take $0\le l\le r$. We endow the subdivided graph $G_r$
  with i.i.d.\ edge weights $W_e>0$, $e\in E_r$. Consider the family $W^{(l)}$ of
  corresponding random weights
from Theorem \ref{thm:vrjp-on-G-r-pos-rec}, all realized on the same probability space
with expectation operator $\E$. We abbreviate
$C_\alpha:=2^{-\alpha}\Gamma(\tfrac12-\alpha)/\sqrt{\pi}$
for $0\le\alpha<\frac12$ and $\cnull:=\gamma+\log 2=1.27036\ldots$, where 
\begin{align}
  \label{eq:euler-mascheroni}
  \gamma=-\int_0^\infty e^{-t}\log t\, dt=0.57721\ldots
\end{align}
denotes the Euler Mascheroni constant. 
For all $\bar e\in E_l$ and $e'\in E_r$, one has
  \begin{align}
    \label{eq:phase1-alpha}
    \E[(W^{(l)}_{\bar e})^\alpha]
    &\le (2^{-\alpha})^{r-l}\E[W_{e'}^\alpha]
    &\hspace{-1cm}\text{ for } \alpha\in[0,1],\\
    \label{eq:combined-alpha}
    \E[(W^{(l)}_{\bar e})^\alpha]
     &\le \frac{1}{C_\alpha}\min_{m\in\{l,\ldots,r\}}\left(C_\alpha 
       (2^{-\alpha})^{r-m}\E[W_{e'}^\alpha]\right)^{2^{m-l}}
     &\hspace{-1cm}\text{ for } \alpha\in[0,\tfrac12),
    \\
      \label{eq:combined-log}
    \E[\log W^{(l)}_{\bar e}]
    &\le \min_{m\in\{l,\ldots,r\}}
      2^{m-l}\left(\E[\log W_{e'}]-(r-m)\log 2+\cnull\right)-\cnull.
  \end{align}
  Set $m_0:=r-2-\lfloor\alpha^{-1}\log_2(C_\alpha \E[W_{e'}^\alpha])\rfloor$
  and $m_1:=r-2-\lfloor(\log 2)^{-1}(\E[\log W_{e'}]+\cnull)\rfloor$.
  If $m_0\in\{l,\ldots,r\}$, a minimizer in \eqref{eq:combined-alpha}
  is given by $m=m_0$. If $m_0<l$ or $m_0>r$, it is given by
  $m=l$ or $m=r$, respectively. The analogous statement holds for $m_1$
  with \eqref{eq:combined-alpha} replaced by \eqref{eq:combined-log}. 
\end{theorem}
Note that for $\alpha\in[0,\frac12)$, the bound \eqref{eq:combined-alpha}
implies the bound \eqref{eq:phase1-alpha}, using $m=l$. 
The proof of Theorem \ref{thm:from-r-to-l} is done in
Section \ref{sec:proofs-for-subdivided-graphs} and given by induction.
The induction step is based on the following lemma.

\begin{lemma}[Induction step for moments of effective weights]
  \label{le:bounds-W-l-one-step}
  \mbox{}\\
  Consider the setup of Theorem \ref{thm:from-r-to-l}. 
  For $l\in\{1,\ldots,r\}$, $\bar e\in E_{l-1}$, and $e'\in E_l$, 
  we have
  \begin{align}
    \label{eq:bound-exp-w-l-alpha}
    &\E[(W^{(l-1)}_{\bar e})^\alpha]\le 2^{-\alpha}\E[(W^{(l)}_{e'})^\alpha]
      \quad\text{for }\alpha\in[0,1],\\
    \label{eq:bound2-exp-w-l-alpha}
&\E[(W^{(l-1)}_{\bar e})^\alpha]\le C_\alpha \E[(W^{(l)}_{e'})^\alpha]^2
      \quad\text{for }\alpha\in[0,\tfrac12),\\
    \label{eq:bound2-exp-w-l-log}
    &\E[\log W^{(l-1)}_{\bar e}]\le\min\{\E[\log W^{(l)}_{e'}]-\log 2,2\E[\log W^{(l)}_{e'}]+\cnull\}
  \end{align}
  with the constants $C_\alpha$ and $\cnull$ from Theorem \ref{thm:from-r-to-l}.
  For $\alpha\in[0,\frac12)$,
  the bound  in \eqref{eq:bound-exp-w-l-alpha} is stronger than
  the bound in \eqref{eq:bound2-exp-w-l-alpha} if and only if
  $\E[(W^{(l)}_{e'})^\alpha]>2^{-\alpha}C_\alpha^{-1}$.
  Similarly, the minimum in bound \eqref{eq:bound2-exp-w-l-log} equals
  $\E[\log W^{(l)}_{e'}]-\log 2$ if and only if
  $\E[\log W^{(l)}_{e'}]\ge -\log 2-\cnull$.
  Note that the expectations do not depend on the choice of $\bar e$ and
  $e'$. 
\end{lemma}

\paragraph{Discussion.}
The iteration of the bound \eqref{eq:bound-exp-w-l-alpha} gives
an exponentially decreasing upper bound for $\E[(W^{(l)}_{\bar e})^\alpha]$
as a function of $r-l$, while iteration of the other bound
\eqref{eq:bound2-exp-w-l-alpha} is only useful for small values of
$\E[(W^{(l)}_{\bar e})^\alpha]$, but then gives a doubly exponentially fast
decreasing bound. Thus, for $0<\alpha<\frac12$, there is a change of regimes in these upper
bounds for $\E[(W^{(l)}_{\bar e})^\alpha]$, consisting of exponential
decay for the first iteration steps with
a transition to doubly exponential decay for later steps. 
One may speculate that there might be a change of regimes 
for the decay of $\E[(W^{(l)}_{\bar e})^\alpha]$ as well, not just for the upper bounds.

\subsection{Non-linear hyperbolic supersymmetric sigma model}
\label{subsec:Non-linear-hyperbolic-supersymmetric-sigma-model}

Let $\Lambda$ be a finite set containing a pinning point $\rho$ and consider
interactions $W=(W_{ij})_{i,j\in\Lambda}$, 
$W_{ij}=W_{ji}\ge0$, such that the graph $(\Lambda,E_+)$ with edge set
$E_+:=\{\{i,j\}\subseteq\Lambda:W_{ij}>0\}$ is connected.

The non-linear hyperbolic supersymmetric sigma model, $\htwo$-model for
short, is a statistical mechanics type model involving spin variables
taking values in a supermanifold called $\htwo$. The spin variables
have three even (= commuting) components $x,y,z$ and two odd (= anticommuting)
components $\xi,\eta$ in a real Grassmann algebra $\cA=\cA_0\oplus\cA_1$
with $\R\subseteq\cA_0$.
Here, $\cA_0\ni x,y,z$ denotes the even subalgebra and $\cA_1\ni\xi,\eta$
the odd subspace. More details can be found in
\cite{disertori-spencer-zirnbauer2010}, \cite{phdthesis-swan2020}, and
\cite[Appendix]{disertori-merkl-rolles2020}.
To every vertex $i\in\Lambda$ linearly independently, we associate a spin variable
$\sigma_i=(x_i,y_i,z_i,\xi_i,\eta_i)$ subject to the constraint
\begin{align}
  \sigma_i\in\htwo:=
  &\{(x,y,z,\xi,\eta)\in\cA_0^3\times\cA_1^2:\, x^2+y^2-z^2+2\xi\eta=-1,\body(z)>0\}.
\end{align}
Here, $\body(z)\in\R$ is the unique real number such that $z-\body(z)$ is
nilpotent. We endow $\cA_0^3\times\cA_1^2$ with the inner product
\begin{align}
\label{eq:def-inner-product}
\sk{\sigma,\sigma'}:=xx'+yy'-zz'+\xi\eta'-\eta\xi'
\end{align}
for $\sigma=(x,y,z,\xi,\eta),\sigma'=(x',y',z',\xi',\eta')$.
For any smooth function $f:\R^k\to\R$ there is an extension to a superfunction 
$f:\cA_0^k\to\cA_0$ constructed by a Taylor expansion in the nilpotent parts;
it is denoted by the same symbol. The same holds if $f$ is defined only on 
an open subset $U$ of $\R^k$, but then the extension is only defined on 
the subset of $\cA_0^k$ with bodies in $U$. 
In particular, on $\htwo$ the component $z$ is not an independent variable,
but just an abbreviation $z=\sqrt{1+x^2+y^2+2\xi\eta}$. 
In the $\htwo$-model, the pinning point $\rho\in\Lambda$ gets constant
spin $\sigma_\rho=o:=(0,0,1,0,0)\in H^{2|2}$ assigned to it. 
The superintegration form $\cD \sigma$ on $H^{2|2}$ is 
defined by 
\begin{align}
f\mapsto\int_{H^{2|2}}\cD \sigma f(\sigma):=
\frac{1}{2\pi}\int_{\R^2}dx\, dy\,
\partial_\xi\partial_\eta\left(\frac{1}{z}f(x,y,z,\xi,\eta)\right)
\end{align}
for any superfunction $f$ 
decaying sufficiently fast to make the integral well-defined. 
The 
$\htwo$-model $\Lambda$ is given by
\begin{align}
  \label{model:sigma-Lambda2}
  \mu^W_\Lambda(\sigma_\Lambda):=\mu^W_{\Lambda,\rho}(\sigma_\Lambda):=\delta_o(d\sigma_\rho)
  \cD\sigma_{\Lambda\setminus\{\rho\}}\,
  \exp\left(\frac12\sum_{i,j\in\Lambda}W_{ij}(1+\sk{\sigma_i,\sigma_j})
  \right).
\end{align}
Here, $\delta_o$ denotes the Dirac measure in $o$. 
Note that $\mu^W_{\Lambda,\rho}$ depends on the choice of $\rho$ due to the
constraint $\sigma_\rho=o$, while the law $\nu^W_\Lambda$ of the
$\beta$-field does not. 

The following result shows that the restriction of the $\htwo$ model
is a mixture of $\htwo$ models. 

\begin{theorem}[Effective weights for restrictions to subsets]
  \label{thm:mixture-of-htwo}
  Let $\Lambda=I\cup J$, $I\cap J=\emptyset$, with $|J|\ge 2$ and $\rho\in J$.
  Using the weights $W^J(\beta_I)$ defined in \eqref{def:W-J} and $\widehat W$ obtained
  from $W$ by wiring all points in $J$ at $\rho$, cf.\ \eqref{eq:def-W}, one has 
  \begin{align}
    \int_{(\htwo)^\Lambda}\mu_\Lambda^W(\sigma_\Lambda)f(\sigma_J)=
    &
      \int_{\R^\Lambda}\nu_\Lambda^W(d\beta)\int_{(\htwo)^J}\mu_J^{W^J(\beta_I)}(\sigma_J)f(\sigma_J)\nonumber\\
    =& \int_{\R^{I\cup\{\rho\}}}\nu_{I\cup\{\rho\}}^{\widehat W}(d\beta)\int_{(\htwo)^J}\mu_J^{W^J(\beta_I)}(\sigma_J)f(\sigma_J)
         \label{eq:mixture-of-h22}
    \end{align}
    for any superfunction $f$ on $(\htwo)^J$ which is compactly supported or decays
    at least sufficiently fast so that the left-hand side of
    \eqref{eq:mixture-of-h22} is well-defined. 
  \end{theorem}

\paragraph{How this paper is organized.}
Section \ref{subsec:Jump-processes} deals with the restriction
of Markov jump processes on finite graphs to subgraphs, the removal of self-loops,
and the combination of these two operations. In Section \ref{subsec:Application-to-vrjp},
this is used as an ingredient to treat
the same operations for vrjp, which is viewed as a mixture of Markov jump processes. 
In particular, Theorem \ref{thm:mixture-of-vrjps} is proved there.
This theory is applied to subdivided graphs in Section \ref{sec:proofs-for-subdivided-graphs}. 
Section \ref{sec:Recursion-for-the-weights} deals with a recursive description of the random
weights $W^{(l)}$ introduced in Theorem \ref{thm:vrjp-on-G-r-pos-rec}. This
results in a proof of Lemma \ref{le:bounds-W-l-one-step} and its consequence 
Theorem \ref{thm:from-r-to-l}. Section \ref{subsec:application-to-vrjp-and-errw}
proves the recurrence statements for vrjp and errw given in 
Theorems \ref{thm:vrjp-on-G-r-pos-rec} and~\ref{thm:Errw-on-subdivided-graphs}.
We avoided using the $\htwo$ model and supersymmetry in
Sections \ref{sec:Representation-as-a-mixture}
and \ref{sec:proofs-for-subdivided-graphs} to make the proofs more
accessible to probabilists. Alternatively, one could deduce Theorem
\ref{thm:mixture-of-vrjps} from the result on the $\htwo$ model given in
Theorem \ref{thm:mixture-of-htwo} instead of using the restriction
and conditioning property of the $\beta$-field. 
Proofs using superspin variables
are confined to Section \ref{sec:susy-proofs}, which proves Theorem
\ref{thm:mixture-of-htwo}. 
In Appendix \ref{sec:facts-IG}, we collect relevant results about
the inverse Gaussian distribution. 
The constants $C_\alpha$, $\cdrei$, and $\cnull$ keep their meaning throughout the paper. 

\section{Representation as a mixture}
\label{sec:Representation-as-a-mixture}

\subsection{Markov jump processes}
\label{subsec:Jump-processes}
\paragraph{Notation.}
Consider a Markov jump process $(X,T)$ on a finite connected graph $G=(\Lambda,E)$
with $|\Lambda|\ge 2$, transition rates $q=(q_{ij})_{i,j\in\Lambda}$, 
and starting point $\rho$. Assume in addition that it is reversible 
with reversible measure $\pi=(\pi_i)_{i\in\Lambda}\in(0,\infty)^\Lambda$,
meaning that $\pi_iq_{ij}=\pi_jq_{ji}$ for all $i,j\in\Lambda$. 
The corresponding discrete-time process $X$ is a
reversible Markovian random walk on the graph $G$ with edge weights, also called
conductances, given by $C_{ij}=C_{ji}=\pi_iq_{ij}$. We assume that $C_{ij}>0$
whenever $\{i,j\}\in E$. 
In this context, the law of the Markov jump process
together with its reversible measure is equivalently parametrized by the
conductance matrix $C=(C_{ij})_{i,j\in\Lambda}$ and $\pi$ instead
of $q$ and $\pi$. In the following, we realize $(X,T)$ as canonical process and
denote its law by
$Q^{C,\Lambda}_{\rho,\pi}$, where $\rho$ denotes the starting point. For $i\in\Lambda$, we denote
the corresponding total transition rate and total weight, respectively, by 
\begin{align}
  \label{eq:def-q-i-C-i}
  q_i:=\sum_{k\in\Lambda}q_{ik}, \quad
  C_i:=\sum_{k\in\Lambda}C_{ik}. 
\end{align}
For $n\in\N_0$ and $i\in\Lambda$, given $(X_l)_{l\le n}$, $(T_l)_{l\le n-1}$, and $X_n=i$,
the random variables $X_{n+1}$ and $T_n$ are conditionally independent. In particular, 
the conditional law of $X_{n+1}$ is specified by
\begin{align}
 \label{eq:one-step-proba-X-T}
  Q^{C,\Lambda}_{\rho,\pi}(X_{n+1}=j|(X_l)_{l\le n}, (T_l)_{l\le n-1}, X_n=i)
  =\frac{q_{ij}}{q_i}
  =\frac{C_{ij}}{C_i}
  =:p_{ij}
\end{align}
for $j\in\Lambda$ 
and the conditional law of $T_n$ is exponential with parameter $q_i$. 
Note that even stronger, the process $X$ and $T_n$ are conditionally independent
under the same condition. 

When we deal only with the discrete-time process $X$, but not with the sequence
of waiting times $T$,
the reversible measure $\pi$ becomes irrelevant; by abuse of notation
we write $Q^{C,\Lambda}_{\rho}$ instead of $Q^{C,\Lambda}_{\rho,\pi}$ in this context. 

Next, we deal with the law of the process $(X^{\neq},T^{\neq})$ with
self-loops removed and of the restriction $(X^J,T^J)$ to a vertex subset $J$
introduced in Definition \ref{def:restriction-to-a-subset}. The following definition
describes the corresponding parameters. 

\begin{definition}[Removal of self-loops and restriction to a subset: parameters]
  \label{def:parameters-restriction-to-a-subset}
  We define new transition probabilities, rates, and weights as follows
  \begin{align}
    & p^{\neq}=\left(p^{\neq}_{ij}:=\frac{p_{ij}}{1-p_{ii}}1_{\{i\neq j\}}\right)_{i,j\in\Lambda}, \quad
      q^{\neq}=(q^{\neq}_{ij}:=q_{ij}1_{\{i\neq j\}})_{i,j\in\Lambda}, \\
    & C^{\neq}=(C^{\neq}_{ij}:=C_{ij}1_{\{i\neq j\}})_{i,j\in\Lambda}, \quad
      C^{\neq}_i:=\sum_{k\in\Lambda}C^{\neq}_{ik}, i\in\Lambda. 
  \end{align}
  For any subset $J\subseteq\Lambda$ with $\rho\in J$, $|J|\ge 2$, we set
  $I=\Lambda\setminus J$ and define 
  \begin{align}
    \label{eq:def-p-J}
    p^J=(p^J_{ij})_{i,j\in J}:=p_{JJ}+\sum_{l=0}^\infty p_{JI}p_{II}^lp_{IJ},
  \end{align}
  which is a convergent series with $p_{ij}^J=Q^{C,\Lambda}_i(X^J_1=j)$.
  Furthermore, we define
  \begin{align}
    C^J=(C^J_{ij}:=C_ip^J_{ij})_{i,j\in J}, \quad
    q^J=\left(q^J_{ij}:=\frac{C_i}{\pi_i}p^J_{ij}\right)_{i,j\in J}.
  \end{align}
  The notation $p^{J\neq}$, $q^{J\neq}$, and $C^{J\neq}$ means that the two
  operations $\mbox{}^J$ and $\mbox{}^{\neq}$ have been applied successively. 
\end{definition}

The next lemma shows that the just defined quantities indeed parametrize 
the laws of the processes $(X^{\neq},T^{\neq})$, $(X^J,T^J)$, and $(X^{J\neq},T^{J\neq})$. 

\begin{lemma}[Laws of removal of self-loops and restriction to a subset]
  \label{le:law-restriction}
  \mbox{}\\
  Consider a reversible Markov jump process $(X,T)$ on the
  finite graph $G$. Assume that it starts in
  $\rho\in J$, has the reversible measure $\pi$ with 
  $\pi_i>0$ for all $i\in\Lambda$, and that the  
  jump rates are given by $q_{ij}=C_{ij}/\pi_i$ for $i,j\in\Lambda$. 
  Then, the processes $(X^{\neq},T^{\neq})$ and $(X^J,T^J)$ are
  again reversible Markov jump processes with rates $q^{\neq}$ and $q^J$, weights
  $C^{\neq}$ and $C^J$, transition probabilities $p^{\neq}$ and $p^J$, and
  reversible measures $\pi$ and $\pi|_J$, respectively. Applying both
  transformations successively, $(X^{J\neq},T^{J\neq})$ is also a reversible Markov jump process
  with rates $q^{J\neq}$, weights $C^{J\neq}$, transition probabilities $p^{J\neq}$, and
  reversible measure $\pi|_J$. 
  In particular, the restriction $X^{J\neq}$ of $X$ to $J$ with self-loops removed
  with starting point $X_0=\rho\in J$ has the same law as
  a random walk on the complete graph over $J$ endowed with the weights $C^{J\neq}$.
  In other words, $Q_{\rho}^{C,\Lambda}(X^{J\neq}\in\cdot)=Q_{\rho}^{C^{J\neq},J}$.
\end{lemma}
\begin{proof}
  Consider the filtration $\F_n=\sigma((X_l)_{l\le n},(T_l)_{l\le n-1})$, $n\in\N_0$.
  When $(X,T)$ is replaced by $(X^{\neq},T^{\neq})$ and $(X^J,T^J)$, the corresponding
  filtrations are denoted by $(\F_n^{\neq})_n$ and $(\F_n^J)_n$, respectively. 

  We treat the process $(X^J,T^J)$ first. 
  Fix $m,n\in\N_0$ and $i\in J$. Let $B^J_{n,m}:=\{X_n^J=i,\tau_n=m\}$.
  Observe that on $B^J_{n,m}$, one has $X_m=i$ and $T_n^J=T_m$ and that
  given $\F_m$, the process $X$ and $T_n^J$ are conditionally independent. Hence,
  conditionally on the same, $X^J_{n+1}$ and $T_n^J$ are independent. Moreover,
  still under the same conditions,   
  $T_n^J=T_m$ is exponentially distributed with parameter $q_i=\sum_{k\in\Lambda}q_{ik}$ and
  for any $j\in J$, the event that
  $X_{n+1}^J=j$ holds with conditional probability $p_{ij}^J$. Note that every event
  $A\in\sigma(\F_n^J,B^J_{n,m})$ with $A\subseteq B^J_{n,m}$ fulfills $A\in\F_m$, and that
  $B^J_{n,m}\in\F_m$ holds. 
  Thus, for $t\ge 0$, on the event $B^J_{n,m}$, one has 
  \begin{align}
    Q^{C,\Lambda}_{\rho,\pi}(X^J_{n+1}=j,T_n^J\ge t|\F_n^J,B^J_{n,m})
    =&E_{Q^{C,\Lambda}_{\rho,\pi}}[Q^{C,\Lambda}_{\rho,\pi}(X^J_{n+1}=j,T_m\ge t|\F_m)|\F_n^J,B^J_{n,m}]
       \nonumber\\
    =&Q^{C,\Lambda}_{\rho,\pi}(X^J_{n+1}=j|\F_n^J,B^J_{n,m})e^{-tq_i}.
  \end{align}
  Summing over $m\in\N_0$ and using the $\sigma$-field
  $\B_n^J:=\sigma(B^J_{n,m},m\in\N_0)$ yields the following on the event
  $\{X_n^J=i\}=\bigcup_{m=0}^\infty B^J_{n,m}$:
  \begin{align}
    Q^{C,\Lambda}_{\rho,\pi}(X^J_{n+1}=j,T_n^J\ge t|\F_n^J,\B^J_n)
    =&Q^{C,\Lambda}_{\rho,\pi}(X^J_{n+1}=j|\F_n^J,\B^J_n)e^{-tq_i}.
  \end{align}
  Conditioning this on the smaller $\sigma$-field $\F_n^J$, we obtain
  on the event $\{X_n^J=i\}$, 
  \begin{align}
    Q^{C,\Lambda}_{\rho,\pi}(X^J_{n+1}=j,T_n^J\ge t|\F_n^J)
    =&Q^{C,\Lambda}_{\rho,\pi}(X^J_{n+1}=j|\F_n^J)e^{-tq_i},
  \end{align}
    and we conclude that $X^J_{n+1}$ and $T_n^J$ are conditionally independent 
  on the event $\{X_n^J=i\}$ given $\F_n^J$. Since the graph
  $G=(\Lambda,E)$ is finite and connected and $C_{ij}>0$ whenever $\{i,j\}\in E$, we find 
  that $\sum_{k\in J}q^J_{ik}=\frac{C_i}{\pi_i}\sum_{k\in J}p^J_{ik}=\frac{C_i}{\pi_i}$
  and hence $p^J_{ij}=\frac{\pi_i}{C_i}q^J_{ij}=q^J_{ij}/\sum_{k\in J}q^J_{ik}$. The
  weights for the restriction fulfill $C^J_{ij}=C_ip^J_{ij}=\pi_iq^J_{ij}$ and the
  reversibility condition $\pi_iq^J_{ij}=C_ip^J_{ij}=C_jp^J_{ji}=\pi_jq^J_{ji}$,
  which can be seen by multiplying the definition \eqref{eq:def-p-J}
  of $p^J_{ij}$ by $C_i$ and using repeatedly the original reversibility condition
  $C_kp_{kl}=C_lp_{lk}$. This proves
  the claim for the process $(X^J,T^J)$.

  Next, we treat the process $(X^{\neq},T^{\neq})$ in a similar way.
  Let $i\in\Lambda$, $m\in\N_0$, and set $B^{\neq}:=\{X_n^{\neq}=i,\sigma_n=m\}$. For
  the rest of this proof, the arguments are understood conditionally on $\F_m$ and $B^{\neq}$.
  Observe that $X_m=i$ and $T_n^{\neq}=\sum_{l=m}^{\sigma_{n+1}-1}T_l$.  
  Inductively on $k\in\N$, on the event $B^{\neq}$, for any
  Borel set $S\subseteq\R^k$ and $j\in\Lambda$, it follows that 
  \begin{align}
    &Q^{C,\Lambda}_{\rho,\pi}(X_m=\ldots=X_{m+k-1}=i, X_{m+k}=j,(T_m,\ldots,T_{m+k-1})\in S|\F_m)\nonumber\\
    =&p_{ii}^{k-1}p_{ij}\operatorname{Exp}(q_i)^{\times k}(S),
       \label{eq:geom-exp}
  \end{align}
  where $\operatorname{Exp}(q_i)^{\times k}$ is the $k$-th power of the exponential distribution
  with parameter $q_i$. 
  Note that only the special case $i=j$ is needed as induction hypothesis in this
  induction.

  Take now $j\neq i$. In this case, \eqref{eq:geom-exp} implies that   
  the waiting time $T_n^{\neq}=\sum_{l=\sigma_n}^{\sigma_{n+1}-1}T_l$ consists of a geometrically distributed  
  number of summands with 
  $Q^{C,\Lambda}_{\rho,\pi}(\sigma_{n+1}-\sigma_n=d|\F_m)=(1-p_{ii})p_{ii}^{d-1}$, $d\in\N$,
  on $B^{\neq}$. 
  Conditioning in addition on $\sigma_{n+1}-\sigma_n$, the summands $T_l$ are
  conditionally i.i.d.\ exponentially distributed with parameter $q_i$. 
  By the thinning property of the Poisson process, the waiting time $T_n^{\neq}$
  is exponentially distributed with parameter
  $(1-p_{ii})q_i=q_i-q_{ii}=\sum_{k\in\Lambda}q_{ik}^{\neq}=:q_i^{\neq}$, where we
  used \eqref{eq:one-step-proba-X-T} for $i=j$.
  Summing over $k$ yields 
  \begin{align}
    \label{eq:explicit-Q-proba}
    Q^{C,\Lambda}_{\rho,\pi}(X_{n+1}^{\neq}=j,T_n^{\neq}\ge t|\F_m)
    =& \sum_{k=1}^\infty p_{ii}^{k-1}p_{ij}\operatorname{Exp}(q_i)^{\ast k}([t,\infty))
       = p_{ij}^{\neq}e^{-t q_i^{\neq}};
  \end{align}
  here $\operatorname{Exp}(q_i)^{\ast k}$
  denotes the $k$-fold convolution of $\operatorname{Exp}(q_i)$. Furthermore, the new transition
  probabilities $p^{\neq}_{ij}$ and the new transition rates $q^{\neq}_{ij}$ are related
  for all $i,j\in\Lambda$ by 
  \begin{align}
    \frac{q^{\neq}_{ij}}{q^{\neq}_i}=\frac{q_{ij}1_{\{i\neq j\}}}{(1-p_{ii})q_i}
    =\frac{p_{ij}1_{\{i\neq j\}}}{1-p_{ii}}=p^{\neq}_{ij}.
  \end{align}
  Finally, the new reversibility relation $C^{\neq}_{ij}=\pi_iq^{\neq}_{ij}=\pi_jq^{\neq}_{ji}$
  is an immediate consequence of the original one. 
\end{proof}

\subsection{Application to vrjp}
\label{subsec:Application-to-vrjp}

In the last section, the parameters $q$, $\pi$, and $C$ were deterministic.
In this section, which deals with vrjp rather than Markov jump processes,
they become random because vrjp in exchangeable time-scale is a mixture
of reversible Markov jump processes as was shown in \cite[Theorem 2]{sabot-tarres2012}.
The role of the deterministic conductances $C_{ij}$ is now
overtaken by random conductances $W_{ij}e^{u_i+u_j}$ with appropriate
random variables $u_i$, $i\in\Lambda$, introduced in Lemma \ref{le:mixture-given-beta-I},
below. These $u$-variables are functions of the $\nu^W_\Lambda$-distributed random
field $\beta$ introduced in \eqref{eq:def-nu}. The following remark
reviews some crucial properties of this $\beta$-field. 

\begin{remark}[Properties of $\beta$,
  {\cite[Sect.\ 5.1, Proposition 1 and Lemma 5]{sabot-zeng15}}]
  \label{rem:restriction-conditioning-property-beta}
\mbox{}\\
 Let $\beta\sim\nu^W_\Lambda$.
 Then, $(\beta_i-\frac12 W_{ii})_{i\in\Lambda}\sim\nu^{\wneq}_\Lambda$
 with $\wneq=(\wneq_{ij}=W_{ij}1_{\{i\neq j\}})_{i,j\in\Lambda}$. 
 For any $i\in\Lambda$, one has
  \begin{align}
    \label{eq:marginal-beta-IG}
    (2\beta_i-W_{ii})^{-1}\sim \ig\left(W_i^{-1},1\right)
        \quad\text{with}\quad
    W_i=\sum_{j\in\Lambda\setminus\{i\}}W_{ij},
  \end{align}
  where $\ig(\mu,\lambda)$ denotes the inverse Gaussian distribution
  with parameters $\mu,\lambda>0$; see Appendix \ref{sec:facts-IG}. 
  Assume that $\Lambda=I\cup J$ is finite with $I\cap J=\emptyset$, $|J|\ge 2$.
  The conditioning property states that conditioned on 
$\beta_I$, one has $\beta_J\sim\nu_J^{W^J(\beta_I)}$, and hence
$(\beta_j-\frac12W_{jj}^J(\beta_I))_{j\in J}\sim\nu_J^{W^{J\neq}(\beta_I)}$
with the weights $W^J$ and $W^{J\neq}$ from \eqref{def:W-J} and \eqref{eq:def-W-J-neq}. 
The restriction property states that 
$\beta_I$ is the restriction of $\beta_{I\cup\{\rho\}}\sim\nu_{I\cup\{\rho\}}^{\widehat W}$
with $\widehat W$ defined in \eqref{eq:def-W}. 
\end{remark}

The next remark describes how to recover the law of the original process
$(X,T)$ with self-loops from its self-loop removed version $(X^{\neq},T^{\neq})$
and additional auxiliary independent Poisson processes. 

\begin{remark}[Decoration of vrjp with self-loops]
  \label{rem:vrjp-self-loops}
  Vrjp with self-loops described by $P_\rho^{W,\Lambda}$ and vrjp without
  self-loops described by $P_\rho^{W^{\neq},\Lambda}$ are related as follows. 
  If $(X,T)$ is distributed according to $P_\rho^{W,\Lambda}$, then
  $(X^{\neq},T^{\neq})$ is distributed according to $P_\rho^{W^{\neq},\Lambda}$.
  Conversely, if $(\tilde X,\tilde T)$ is distributed according to
  $P_\rho^{W^{\neq},\Lambda}$, given any vertex $i\in\Lambda$, we take a Poisson process
  with intensity $\frac12 W_{ii}$, visualized as exponential clocks. 
  These Poisson processes should be independent of each other and $(\tilde X,\tilde T)$.
  Whenever the jumping particle is at $i\in\Lambda$, we include a self-loop
  $i\to i$ whenever the corresponding exponential clock rings. In other words,
  we include self-loops at $i$ with rate $\frac12 W_{ii}$. The resulting augmented
  process $(X,T)$ is then distributed according to $P_\rho^{W,\Lambda}$. 
\end{remark}

The next lemma introduces the $u$-field $u=(u_i)_{i\in\Lambda}$
as a function of the $\beta$-field. It then describes how the $u$-field behaves
under restriction of the underlying vertex set $\Lambda$ to some subset $J\subseteq\Lambda$
with $\rho\in J$. This restriction property is used to express the effective
random conductances $C_{ij}^J$, $i,j\in J$, for the restriction of vrjp to $J$. 
We abbreviate $J_-:=J\setminus\{\rho\}$ and $\Lambda_-:=\Lambda\setminus\{\rho\}$.

\begin{lemma}[Restriction property of the $u$-field]
  \label{le:mixture-given-beta-I}
  Let $\Lambda=I\cup J$ be finite with $I\cap J=\emptyset$, $\rho\in J$, and
  $W\in[0,\infty)^{\Lambda\times\Lambda}$. 
  For $\beta=(\beta_I,\beta_J)\in\R^\Lambda$ such that $[H_\beta^W]_{\Lambda_-\Lambda_-}$
  is positive definite, let 
  $e^{u_{\Lambda_-}(\beta)}=([H_\beta^W]_{\Lambda_-\Lambda_-})^{-1}W_{\Lambda_-\rho}$
  and $u_\rho(\beta)=0$.
  Let $H_{\beta_J}^{W^J(\beta_I)}$ and $H_{\beta^{J\neq}}^{W^{J\neq}(\beta_I)}$
  denote the $J\times J$ matrices obtained
  from $H_\beta^W$ defined in \eqref{eq:def-H-beta} with $(\beta,W)$
  replaced by $(\beta_J,W^J(\beta_I))$ and $(\beta^{J\neq},W^{J\neq}(\beta_I))$,
  respectively, with $\beta^{J\neq}:=(\beta_j-\frac12W^J_{jj}(\beta_I))_{j\in J}$,
  $W^J(\beta_I)$ from \eqref{def:W-J}, and 
  $W^{J\neq}(\beta_I)$ from \eqref{eq:def-W-J-neq}. 
  Then, one has the following restriction property
  \begin{align}
    \label{eq:exp-u-J-in-volume-J}
    e^{u_{J_-}(\beta)}=([H_{\beta_J}^{W^J(\beta_I)}]_{J_-J_-})^{-1}W_{J_-\rho}^J(\beta_I)
    =([H_{\beta^{J\neq}}^{W^{J\neq}(\beta_I)}]_{J_-J_-})^{-1}W_{J_-\rho}^{J\neq}(\beta_I). 
  \end{align}
  As a consequence, $u_{J_-}(\beta)$ depends only on $W^J(\beta_I)$ and $\beta_J$.
  We write $u_J(W^J(\beta_I),\beta_J)=u_J(W^{J\neq}(\beta_I),\beta^{J\neq})$
  instead of $u_J(\beta)$.
  
  As an application of \eqref{eq:exp-u-J-in-volume-J} for fixed $\beta_I$, 
  vrjp on $J$ with parameters $W^J(\beta_I)$ and $W^{J\neq}(\beta_I)$, respectively, are
  mixtures of reversible Markov jump processes with weights 
  \begin{align}
    \label{eq:def-cal-W-J-neq}
     C^J_{ij}(\beta_I,e^{u_J})
    =W^J_{ij}(\beta_I)e^{u_i+u_j} \quad\text{and}\quad
    C^{J\neq}_{ij}(\beta_I,e^{u_J})
    =W^{J\neq}_{ij}(\beta_I)e^{u_i+u_j}, 
    \quad i,j\in J,
  \end{align}
  and the same reversible measure $\pi(u_J)=(2e^{2u_i})_{i\in J}$ 
  with $u_J=u_J(W^J(\beta_I),\beta_J)$ in both cases, where
  $\beta_J$ is a $\nu_J^{W^J(\beta_I)}$-distributed random variable.
  In other words, for any event $A\subseteq J^{\N_0}\times\R_+^{\N_0}$, one has
  \begin{align}
    \label{eq:vrjp-on-J-with-W-J-neq-as-mixture}
  P_\rho^{W^{J\neq}(\beta_I),J}((X,T)\in A)=&
  \int_{\R^J}Q^{C^{J\neq}(\beta_I,e^{u_J(W^{J\neq}(\beta_I),\tilde\beta_J)}),J}_{\rho,\pi(u_J(W^{J\neq}(\beta_I),\tilde\beta_J))}((X,T)\in A)\,
  \nu_J^{W^{J\neq}(\beta_I)}(d\tilde\beta_J)
  \end{align}
  and the same formula with ``$J{\neq}$'' replaced by ``$J$'' at all five
  occurrences.
\end{lemma}

Thus, the following two procedures yield the same result:
\begin{itemize}
\item Deriving the $u$-field on $\Lambda$ and then restricting it to $J_-$. 
\item Taking random weights $W^J(\beta_I)$, depending only on the restriction
  of $\beta$ to $I$ and then using these random weights to derive the
  $u$-field on $J_-$. 
\end{itemize}

\medskip\noindent
\begin{proof}[Proof of Lemma \ref{le:mixture-given-beta-I}]
The defining relation
  $[H_\beta^W]_{\Lambda_-\Lambda_-}e^{u_{\Lambda_-}}=
  W_{\Lambda_-\rho}$ of $u_{\Lambda_-}=u_{\Lambda_-}(\beta)$
  can be rewritten in block diagonal form
  \begin{align}
    \label{eq:defining-relation-exp-u}
    \begin{pmatrix}
      [H_\beta^W]_{II} & -W_{IJ_-}\\
      -W_{J_-I} & [H_\beta^W]_{J_-J_-}
    \end{pmatrix}
                  \begin{pmatrix}
                    e^{u_I}\\ e^{u_{J_-}}
                  \end{pmatrix}
    =
    \begin{pmatrix}
      W_{I\rho}\\ W_{J_-\rho}
    \end{pmatrix}
.
  \end{align}
 Multiplying the first equation from the left with 
$W_{J_-I}([H_\beta^W]_{II})^{-1}$, we obtain
\begin{align}
  \label{eq:first-def-rel-exp-u}
  W_{J_-I}e^{u_I}
  -W_{J_-I}([H_\beta^W]_{II})^{-1}W_{IJ_-}e^{u_{J_-}}
  =W_{J_-I}([H_\beta^W]_{II})^{-1}W_{I\rho}.
\end{align}
Using the definitions of $H_\beta^W$ and $W^J(\beta_I)$, we calculate
\begin{align}
  \label{eq:H-beta-J-W-J}
  [H_{\beta_J}^{W^J(\beta_I)}]_{J_-J_-}e^{u_{J_-}}
  = [H_\beta^W]_{J_-J_-}e^{u_{J_-}}-W_{J_-I}([H_\beta^W]_{II})^{-1}W_{IJ_-}e^{u_{J_-}}.
\end{align}
Here we used that the $i$-th diagonal element of $H_{\beta_J}^{W^J(\beta_I)}$ equals
$2\beta_i-W^J_{ii}(\beta_I)=2\beta_i-W_{ii}-W_{iI}([H_\beta^W]_{II})^{-1}W_{Ii}$.
The second equation from \eqref{eq:defining-relation-exp-u} yields
\begin{align}
  [H_\beta^W]_{J_-J_-}e^{u_{J_-}}=W_{J_-I}e^{u_I}+W_{J_-\rho}. 
\end{align}
Inserting this in \eqref{eq:H-beta-J-W-J} and then using \eqref{eq:first-def-rel-exp-u},
we obtain
\begin{align}
  \label{eq:e-u-J-in-terms-of-beta-J}
  [H_{\beta_J}^{W^J(\beta_I)}]_{J_-J_-}e^{u_{J_-}}
  =&W_{J_-\rho}+W_{J_-I}e^{u_I}-W_{J_-I}([H_\beta^W]_{II})^{-1}W_{IJ_-}e^{u_{J_-}}
     \nonumber \\
  =&W_{J_-\rho}+W_{J_-I}([H_\beta^W]_{II})^{-1}W_{I\rho}
  =W_{J_-\rho}^J(\beta_I).
\end{align}
Thus, $e^{u_{J_-}}=([H_{\beta_J}^{W^J(\beta_I)}]_{J_-J_-})^{-1}W_{J_-\rho}^J(\beta_I)$,
which proves the first equality in \eqref{eq:exp-u-J-in-volume-J}.
The second equality is an immediate consequence of the definitions of
$H_{\beta^{J\neq}}^{W^{J\neq}(\beta_I)}$ and $\beta^{J\neq}$, cf.\
formula \eqref{eq:def-H-beta}. 
\cite[Theorem 2]{sabot-tarres2012} implies that 
the vrjp on $J$ starting at $\rho$ with parameters $W^{J\neq}(\beta_I)$
is a mixture of Markov jump processes with weights 
$C^{J\neq}(\beta_I,e^{u_J(W^J(\beta_I),\beta_J)})$ given in 
\eqref{eq:def-cal-W-J-neq} and reversible measure
$\pi(u_J(W^J(\beta_I),\beta_J))$ with $\beta_J\sim\nu_J^{W^J(\beta_I)}$
with the given $\beta_I$.
Using $u_J(W^J(\beta_I),\beta_J)=u_J(W^{J\neq}(\beta_I),\beta^{J\neq})$ and
$\beta^{J\neq}\sim\nu_J^{W^{J\neq}(\beta_I)}$, claim 
\eqref{eq:vrjp-on-J-with-W-J-neq-as-mixture} follows.
The variant of \eqref{eq:vrjp-on-J-with-W-J-neq-as-mixture} with ``$J{\neq}$''
replaced by ``$J$'' is then obtained by a decoration with self-loops as described in
Remark \ref{rem:vrjp-self-loops}.
\end{proof}

We now prove that restriction of vrjp to a subset $J\subseteq\Lambda$
containing the starting point is a mixture of vrjps. 

\medskip\noindent
\begin{proof}[Proof of Theorem \ref{thm:mixture-of-vrjps}]
  Let $\beta=\beta_\Lambda$ denote the canonical process on $\R^\Lambda$ with law $\nu_\Lambda^W$.
  On the event that $[H_\beta^W]_{\Lambda_-\Lambda_-}$ is positive definite, which is a
  $\nu_\Lambda^W$-a.s.\ event, let $u_\rho=0$ and
  $e^{u_{\Lambda_-}}=(e^{u_i})_{i\in\Lambda_-}=([H_\beta^W]_{\Lambda_-\Lambda_-})^{-1}W_{\Lambda_-\rho}$
  with $\Lambda_-=\Lambda\setminus\{\rho\}$.
  By \cite[Theorem 2]{sabot-tarres2012},
  the vrjp in exchangeable time scale on $\Lambda$ with initial
  parameters $W$ is a mixture of Markov jump processes with random transition rates
  $q_{ij}(\beta):=\frac12W_{ij}e^{u_j-u_i}$ for any edge $\{i,j\}$. In matrix notation
  we can write this as
  \begin{align}
    \label{eq:def-q-ij}
    (q_{ij}(\beta))_{i,j\in\Lambda}
    =\frac12e^{-u}We^u\text{ with }
    e^{\pm u}=\diag(e^{\pm u_i})_{i\in\Lambda}. 
  \end{align}
  Observe that the defining equation
  of $e^{u_{\Lambda_-}}$ is equivalent to 
  $\beta_i=\frac12\sum_{j\in\Lambda}W_{ij}e^{u_j-u_i}=\sum_{j\in\Lambda}q_{ij}(\beta)$,
  $i\in\Lambda_-$ and consequently, $\beta_i$ plays the random analogue of the total
  jump rate $q_i$, cf.\ formula \eqref{eq:def-q-i-C-i}.
  The random rates are
  equivalently described by the random weights
  $C_{ij}(\beta):=W_{ij}e^{u_i+u_j}$ and the random reversible measure
  $\pi_i(\beta):=2e^{2u_i}$. The random transition probabilities $p_{ij}(\beta)$ can
  now be described in terms of the random total weight 
  $C_i(\beta):=\sum_{j\in\Lambda}C_{ij}(\beta)=2\beta_ie^{2u_i}$ by
  $p_{ij}(\beta)=C_{ij}(\beta)/C_i(\beta)=q_{ij}(\beta)/\beta_i$. Combining this with
  \eqref{eq:def-q-ij} yields the transition probability matrix
  \begin{align}
    \label{eq:p-ij-in-terms-of-W-u}
    p(\beta)=(p_{ij}(\beta))_{i,j\in\Lambda}=e^{-u}(2\beta)^{-1}We^u
    \text{ with }\beta:=\diag(\beta_i)_{i\in\Lambda}. 
  \end{align}
  The assumption that $[H_\beta^W]_{\Lambda_-\Lambda_-}$ is positive definite
  implies $M:=(2\beta_I)^{-\frac12}W_{II}(2\beta_I)^{-\frac12}<\Id$. Since the symmetric matrix
  $M$ has only non-negative entries, the Frobenius theorem implies $-M<\Id$ as well.
  It follows that the geometric series 
  \begin{align}
    \sum_{l=0}^\infty[(2\beta_I)^{-1}W_{II}]^l
    =&(2\beta_I)^{-\frac12}\sum_{l=0}^\infty M^l(2\beta_I)^{\frac12}\nonumber\\
    =&(2\beta_I)^{-\frac12}(\Id-M)^{-1}(2\beta_I)^{\frac12}=([H^W_\beta]_{II})^{-1}2\beta_I
  \end{align}
  converges. 
  Using the definition \eqref{eq:def-p-J} of $p^J(\beta)$, its description 
  \eqref{eq:p-ij-in-terms-of-W-u}, and the last identity, 
  we calculate 
  \begin{align}
    & p^J(\beta)=p_{JJ}(\beta)+\sum_{l=0}^\infty p_{JI}(\beta)p_{II}(\beta)^lp_{IJ}(\beta)
    \nonumber\\
    =&e^{-u_J}(2\beta_J)^{-1}\left[W_{JJ}
       +\sum_{l=0}^\infty W_{JI} [(2\beta_I)^{-1}W_{II}]^l (2\beta_I)^{-1}W_{IJ}\right]e^{u_J}
    \nonumber\\
    =&e^{-u_J}(2\beta_J)^{-1}W^J(\beta_I)e^{u_J}. 
  \end{align}
  Combined with $\diag(C_i(\beta))_{i\in J}=2\beta_Je^{2u_J}$,
  this yields the weights  
  $C^J(\beta)=\diag(C_i(\beta))_{i\in J}\cdot p^J(\beta)=e^{u_J}W^J(\beta_I)e^{u_J}$ for the
  restriction to $J$.   
  By Lemma \ref{le:law-restriction}, for given $\beta$, the restriction of a Markov jump
  process with weights $C(\beta)$ and reversible measure $\pi(\beta)$ to $J$
  with self-loops removed is a reversible Markov jump process with weights
  $C^{J\neq}(\beta_I,e^{u_J})=(C^{J\neq}_{ij}(\beta_I,e^{u_J}):=W^J_{ij}(\beta_I)e^{u_i+u_j}1_{\{i\neq j\}})_{i,j\in J}$
  and reversible measure $\pi(u_J)=(2e^{2u_i})_{i\in J}$. In other words, for any event 
  $A\subseteq J^{\N_0}\times\R_+^{\N_0}$,
\begin{align}
  \label{eq:rel-different-Qs}
  Q^{C(\beta),\Lambda}_{\rho,\pi}((X^{J\neq},T^{J\neq})\in A)
  =Q^{C^{J\neq}(\beta_I,e^{u_J}),J}_{\rho,\pi(u_J)}((X,T)\in A). 
\end{align}
By Lemma \ref{le:mixture-given-beta-I}, 
$u_J=u_J(W^J(\beta_I),\beta_J)=u_J(W^{J\neq}(\beta_I),\beta^{J\neq})$.
For the restriction of vrjp to $J$ with self-loops removed, it follows 
  \begin{align}
    P_\rho^{W,\Lambda}((X^{J\neq},T^{J\neq})\in A)
    =&\int_{\R^\Lambda}Q^{C(\beta),\Lambda}_{\rho,\pi}((X^{J\neq},T^{J\neq})\in A)\,\nu_\Lambda^W(d\beta)\nonumber\\
    =&\int_{\R^\Lambda}Q^{C^{J\neq}(\beta_I,u_J(W^{J\neq}(\beta_I),\beta^{J\neq}),J}_{\rho,\pi(u_J(W^{J\neq}(\beta_I),\beta^{J\neq}))}((X,T)\in A)\,\nu_\Lambda^W(d\beta).
       \label{eq:mixture1}
  \end{align}
We apply the restriction and conditioning property of $\beta\sim\nu^W_\Lambda$
cited in Remark \ref{rem:restriction-conditioning-property-beta}.
The restriction property states that
$\beta_I$ is the restriction of $\beta_{I\cup\{\rho\}}\sim\nu_{I\cup\{\rho\}}^{\widehat W}$
with $\widehat W$ defined in \eqref{eq:def-W}.
By the conditioning property, given $\beta_I$, one has $\beta_J\sim\nu_J^{W^J(\beta_I)}$
and therefore $\beta^{J\neq}\sim\nu_J^{W^{J\neq}(\beta_I)}$.
Consequently, 
\begin{align}
  \label{eq:mixture2}
  \rhs\eqref{eq:mixture1}=
  \int_{\R^{I\cup\{\rho\}}}\int_{\R^J}Q^{C^{J\neq}(\beta_I,e^{u_J(W^{J\neq}(\beta_I),\tilde\beta_J)}),J}_{\rho,\pi(u_J(W^{J\neq}(\beta_I),\tilde\beta_J))}((X,T)\in A)\,
  \nu_J^{W^{J\neq}(\beta_I)}(d\tilde\beta_J)\,\nu_{I\cup\{\rho\}}^{\widehat W}(d\beta_{I\cup\{\rho\}}).
\end{align}
By Lemma \ref{le:mixture-given-beta-I}, the inner integral equals
$P_\rho^{W^{J\neq}(\beta_I),J}((X,T)\in A)$. The second equality in the claim
\eqref{eq:mixture-of-vrjps} follows from the restriction property of the
$\beta$-field. 
The statement for $(X^J,T^J)$ follows analogously, with ``$J{\neq}$'' replaced by
``$J$'' in \eqref{eq:mixture1} and \eqref{eq:mixture2}. This completes the proof of the theorem.
\end{proof}

\section{Proofs for subdivided graphs}
\label{sec:proofs-for-subdivided-graphs}
\subsection{Recursion for the weights}
\label{sec:Recursion-for-the-weights}

In this section, we deal with subdivided graphs, which are obtained by
iteratedly introducing a new vertex in the middle of every edge. Going back
from a subdivided graph to the original graph corresponds then to a restriction.
The next lemma shows a kind of semigroup property for this restriction
operation on effective weights: Iterated restriction on weights gives the
same result as restriction in one single step. 

\begin{lemma}[Effective weights for iterated restrictions to subsets]
  \label{le:mixture-of-htwo2}
 Let $\Lambda=I\cup J$ with disjoint finite sets $I$ and $J$, $\rho\in J$.
  One has 
  \begin{align}
    \label{eq:H-beta-double-inverse}
    ([H_\beta^{-1}]_{JJ})^{-1}=[H_\beta]_{JJ} -W_{JI}([H_\beta]_{II})^{-1}W_{IJ} 
    =2\diag(\beta_J)-W^J(\beta_I)
    =H_{\beta_J}^{W^J(\beta_I)}
  \end{align}
  with the weights $W^J(\beta_I)$ defined in \eqref{def:W-J}. 
  For $\tilde J\subsetneq J$ with $\rho\in\tilde J$ and
    $\tilde I=\Lambda\setminus\tilde J$, one has 
    \begin{align}
      \label{eq:W-tilde-J-W-J}
      W^{{\tilde J}}(\beta_{\tilde I})=(W^J(\beta_I))^{\tilde{J}}(\beta_{J\setminus\tilde J}).
    \end{align}
\end{lemma}
\begin{proof}
Formula \eqref{eq:H-beta-double-inverse} is a special case of the formula for
the Schur complement combined with the definitions of $W^J(\beta_I)$ and 
\eqref{eq:def-H-beta} of $H_\beta=H_\beta^W$. Formula \eqref{eq:W-tilde-J-W-J}
is a consequence of the following calculation, which uses the inversion
formula \eqref{eq:H-beta-double-inverse} twice, first for $H_\beta^W$ and
second for $H_{\beta_J}^{W^J(\beta_I)}$. 
\begin{align}
  &2\diag(\beta_{{\tilde J}})-W^{{\tilde J}}(\beta_{\tilde I})
  =([(H_\beta^W)^{-1}]_{{\tilde J}{\tilde J}})^{-1}
    =([[(H_\beta^W)^{-1}]_{JJ}]_{{\tilde J}{\tilde J}})^{-1}
  \nonumber\\
  &=([(2\diag(\beta_J)-W^J(\beta_I))^{-1}]_{\tilde J \tilde J})^{-1}
    =([(H_{\beta_J}^{W^J(\beta_I)})^{-1}]_{\tilde J \tilde J})^{-1}\nonumber\\
  &=2\diag(\beta_{{\tilde J}})-(W^J(\beta_I))^{\tilde{J}}(\beta_{J\setminus\tilde J}).
    \label{eq:diag-beta-tilde-J-W-tilde-J}
\end{align}
\end{proof}

The next definition introduces subdivisions and some notations associated with
it more formally. In the following, let $G=(\Lambda,E)$ be an undirected graph
without self-loops.

\begin{definition}[Subdivided graphs]
  \label{def:subdivided-graph}
  For $r\in\N_0$, the \emph{$2^r$-subdivision}
$G_r=(\Lambda_r,E_r)$ of $G$ is obtained by replacing every edge by a series
of $2^r$ edges as follows.
\begin{itemize}
\item $\Lambda_r$ is obtained from $\Lambda$ by adding $2^r-1$ new vertices $v_{e,2^{-r}j}$,
  $j=1,\ldots,2^r-1$ for
  any edge $e\in E$. We say that these new vertices are located on the edge $e$. 
  In particular, $\Lambda\subseteq \Lambda_r$. 
\item Every edge $e\in E$ is replaced by a series of $2^r$ edges $e_{1,r},\ldots,e_{2^r,r}$ as follows.
  For bookkeeping purposes only, we endow $e$ with a direction, and 
  call $v_{e,0}$ and $v_{e,1}$ the two vertices it consists of. Then, $e$ is
  replaced by the new edges $e_{j,r}=\{v_{e,2^{-r}(j-1)},v_{e,2^{-r}j}\}$, $1\le j\le 2^r$,
  which we view as being located on the edge $e$. Thus, 
  $E_r=\{e_{j,r}:e\in E, 1\le j\le 2^r\}$.
\end{itemize}
Note that $\Lambda_0=\Lambda$ and $\Lambda_r\subseteq\Lambda_{r+1}$. 
For technical convenience, we introduce also the set of
direct self-loops $\Delta_r:=\{\{i,i\}:i\in\Lambda_r\}$. 
\end{definition}

From now on, we fix $r\in\N_0$. 
We endow the graph $G_r=(\Lambda_r,E_r)$ with 
strictly positive possibly random edge weights $W=(W_e)_{e\in E_r}$.
Given $W$, let $\beta=(\beta_v)_{v\in\Lambda_r}$ have the conditional distribution
$\nu_{\Lambda_r}^W$. Although this probability measure is not only defined
for finite graphs, but also for infinite locally finite ones, let us assume for the
moment that the graph $G$ is finite. 
For varying $l\in\{0,\ldots,r\}$, consider the effective weights
\begin{align}
  W^{(l)}:=
  W^{\Lambda_l}=W^{\Lambda_l}(\beta_{\Lambda_r\setminus\Lambda_l})=(W_e^{\Lambda_l})_{e\in E_l\cup\Delta_l}
\end{align}
on $G_l$ using the notation of \eqref{def:W-J}. We abbreviate
  \begin{align}
    \label{eq:def-hat-beta}
    \beta^{(l)\neq}_v:=\beta^{\Lambda_l\neq}=\beta_v-\frac12W_{vv}^{\Lambda_l},v\in\Lambda_l, 
  \end{align}
  cf.\ Lemma \ref{le:mixture-given-beta-I}. 
  In particular, $W^{(r)}=W$ and $\beta^{(r)\neq}=\beta$.

\begin{lemma}[Recursion for the random weights]
  \label{le:aux-W-l-1}
  \mbox{}\\
  Let the graph $G$ be finite and $l\in\{0,\ldots,r\}$. 
  \begin{enumerate}
  \item \label{item0}
    Given $(W^{(l)}_e)_{e\in E_l}$, the random vector $\beta^{(l)\neq}$ has the
  conditional distribution $\nu_{\Lambda_l}^{W^{(l)\neq}}$. 
  \item \label{item1}
    Assuming $l\ge 1$, take an arbitrary edge $\bar e=e_{j,l-1}\in E_{l-1}$ with
      $e\in E$ and $1\le j\le 2^{l-1}$; see Figure \ref{fig:splitting-an-edge}.
      Call $v=v_{e,2^{-l}(2j-1)}\in \Lambda_l\setminus\Lambda_{l-1}\subseteq\Lambda_l\setminus\Lambda_0$
      its midpoint.
      It splits $\bar e$ into the two edges $e'=e_{2j-1,l}$ and $e''=e_{2j,l}$ in $E_l$. 
      Then, one has 
      \begin{align}
        \label{eq:recursion-W-l}
   W^{(l-1)}_{\bar e}=\frac{W_{e'}^{(l)}W_{e''}^{(l)}}{2\beta_v^{(l)\neq}}.
 \end{align}
\item \label{item2}
  Assume $l\ge 2$, which implies that $\Lambda_{l-1}\setminus\Lambda_0\neq\emptyset$. 
  Take an arbitrary vertex
      $\bar v=v_{e,2^{-(l-1)}j}\in\Lambda_{l-1}\setminus\Lambda_0$ with $e\in E$ and
      $1\le j\le 2^{l-1}-1$. In particular, $\bar v$ belongs also to
      $\Lambda_l\setminus\Lambda_0$; see Figure \ref{fig:splitting-an-edge}.
      Call $e^1=e_{2j,l}\in E_l$ and $e^2=e_{2j+1,l}\in E_l$ the
      two edges adjacent to $\bar v$ in $G_l$. Call
      $v^1=v_{e,2^{-l}(2j-1)},v^2=v_{e,2^{-l}(2j+1)}\in\Lambda_l\setminus\Lambda_{l-1}\subseteq\Lambda_l\setminus\Lambda_0$
      the two neighboring vertices
      of $\bar v$ in $G_l$. Then, one has 
      \begin{align}
        \label{eq:recursion-beta-hat-v}
        2\beta_{\bar v}^{(l-1)\neq}=2\beta_{\bar v}^{(l)\neq}-\frac{(W_{e^1}^{(l)})^2}{\beta_{v^1}^{(l)\neq}}
      -\frac{(W_{e^2}^{(l)})^2}{\beta_{v^2}^{(l)\neq}}.
    \end{align}
  \item
    \label{item3} Assume $l\ge 1$ again. 
    Given $(W^{(l)}_e)_{e\in E_l}$, the random variables $\beta^{(l)\neq}_v$,
    $v=v_{e,2^{-l}(2j-1)}\in\Lambda_l\setminus\Lambda_{l-1}$, are conditionally independent with
    conditional law $(2\beta^{(l)\neq}_v)^{-1}\sim \ig((W_{e'}^{(l)}+W_{e''}^{(l)})^{-1},1)$, where
    we use the notation $e'=e_{2j-1,l}$ and $e''=e_{2j,l}$ from item \ref{item1}.
  \item \label{item5}
    If the family of weights $(W_e)_{e\in E}$ is independent or even i.i.d.,
    then so is the family of weights $(W^{(l)}_e)_{e\in E_l}$.
  \end{enumerate}
\end{lemma}
\begin{proof}
  Claim \ref{item0} is a consequence of the conditioning property cited in
  Remark \ref{rem:restriction-conditioning-property-beta}.

  We prove now items \ref{item1} and \ref{item2}. 
  Locally for this proof, we abbreviate $J=\Lambda_{l-1}$ and 
  $I=\Lambda_l\setminus\Lambda_{l-1}$. We observe that any edge $e\in E_l$
  connects a vertex in $J$ to a vertex in $I$; see Figure \ref{fig:splitting-an-edge}. In particular, there
are neither edges in $E_l$ between two vertices in $J$ nor
between two vertices in $I$. As a consequence, 
$W^{(l)}_{JJ}=\diag(W^{(l)}_{vv})_{v\in J}$ and
$[H_{\Lambda_l,\beta}^{W^{(l)}}]_{II}
=\diag(2\beta_v-W^{(l)}_{vv})_{v\in I}=\diag(2\beta_v^{(l)\neq})_{v\in I}=2\beta_I^{(l)\neq}$.  
Using first \eqref{eq:W-tilde-J-W-J} and then formula \eqref{def:W-J}, it follows 
\begin{align}
  W^J=(W^{(l)})^{J}
  =&W^{(l)}_{JJ}+W^{(l)}_{JI}
                      ([H_{\Lambda_l,\beta}^{W^{(l)}}]_{II})^{-1}W^{(l)}_{IJ}\nonumber\\
  =&\diag(W^{(l)}_{vv})_{v\in J}+
  W^{(l)}_{JI}(2\beta_I^{(l)\neq})^{-1}     W^{(l)}_{IJ}.
\end{align}
This implies formula \eqref{eq:recursion-W-l} when reading it for
off-diagonal entries and 
\begin{align}
        \label{eq:recursion-Wvv-new}
        W^{(l-1)}_{\bar v\bar v}=W_{\bar v\bar v}^{(l)}+\frac{(W_{e^1}^{(l)})^2}{2\beta_{v^1}^{(l)\neq}}
        +\frac{(W_{e^2}^{(l)})^2}{2\beta_{v^2}^{(l)\neq}}.
\end{align}
when reading it on the diagonal. Formula \eqref{eq:recursion-beta-hat-v} follows. 

We prove now item \ref{item3}. 
There is no edge in $E_l$ connecting two vertices in $\Lambda_l\setminus\Lambda_{l-1}$.
The following arguments are understood conditionally on $(W^{(l)}_e)_{e\in E_l}$.
The one-dependence of
$\beta^{(l)\neq}\sim\nu_{\Lambda_l}^{W^{(l)\neq}}$ implies that the entries
$\beta^{(l)\neq}_v$, $v\in\Lambda_l\setminus\Lambda_{l-1}$ are independent.
Furthermore, the inverse Gaussian law of $(\beta^{(l)\neq}_v)^{-1}$ follows
from \eqref{eq:marginal-beta-IG}. 

Finally, we prove claim \ref{item5}. Assume that $(W_e)_{e\in E}$ are independent
or i.i.d., respectively. We prove
that $W^{(l)}_e$, $e\in E_l$, are unconditionally independent or i.i.d., respectively,
by induction over
$l=r,r-1,\ldots,0$. The initial case $l=r$ holds by assumption.
For the induction step $l\to l-1$ we use the recursion relation
\eqref{eq:recursion-W-l}. 
Given $(W^{(l)}_e)_{e\in E_l}$, the reciprocal denominators 
$(2\beta^{(l)\neq}_v)^{-1}$
with $v=v_{e,2^{-l}(2j-1)}$ for $1\le j\le 2^{l-1}$ and $e\in E$ are conditionally
independent with conditional law $\ig((W_{e_{2j-1,l}}^{(l)}+W_{e_{2j,l}}^{(l)})^{-1},1)$.
The parameter $(W_{e_{2j-1,l}}^{(l)}+W_{e_{2j,l}}^{(l)})^{-1}$ and the
numerator $W_{e_{2j-1,l}}^{(l)}W_{e_{2j,l}}^{(l)}$ both are functions of the pair
$(W_{e_{2j-1,l}}^{(l)},W_{e_{2j,l}}^{(l)})$. As $j$ runs from $1$ to $2^{l-1}$
and $e$ runs through $E$, these pairs are independent or i.i.d., respectively,
by induction hypothesis. 
In view of \eqref{eq:recursion-W-l} this concludes the induction step. 
\end{proof}

\begin{figure}
  \centering
    \includegraphics[width=\textwidth]{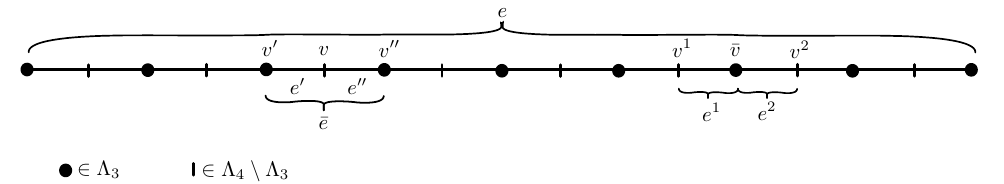}
  \caption{Illustration of the recursion $l\leadsto l-1$ for $l=4$ in the cases
    $\bar e=e_{3,3}$, $e'=e_{5,4}$, $e''=e_{6,4}$, $v'=v_{e,\frac28}$, $v=v_{e,\frac{5}{16}}$,
    $v''=v_{e,\frac38}$,
    $e^1=e_{12,4}$, $e^2=e_{13,4}$, $v^1=v_{e,\frac{11}{16}}$,
    $\bar v=v_{e,\frac68}$, and $v^2=v_{e,\frac{13}{16}}$.}
  \label{fig:splitting-an-edge}
\end{figure}

This lemma is an important ingredient in the proof of the induction step
for the moments of $W_e^{(l)}$. 

\medskip\noindent
\begin{proof}[Proof of Lemma \ref{le:bounds-W-l-one-step}]
  Take $\bar e\in E_{l-1}$. Let $v$ be the midpoint of $\bar e$. It 
  splits it into two edges, which we may call $e',e''\in E_l$; cf.\ item \ref{item1}
  in Lemma \ref{le:aux-W-l-1}.
  Using the recursion equation \eqref{eq:recursion-W-l}, one has 
  \begin{align}
    \E[(W^{(l-1)}_{\bar e})^\alpha]
    =&\E\left[\left(\frac{W_{e'}^{(l)}W_{e''}^{(l)}}{2\beta_v^{(l)\neq}}\right)^\alpha\right]
       =\E[(W_{e'}^{(l)}W_{e''}^{(l)})^\alpha
       \E[(2\beta_v^{(l)\neq})^{-\alpha}|W^{(l)}]]
       \label{eq:exp-W-hoch-alpha-induction-step}, \\
    \E[\log W^{(l-1)}_{\bar e}]
  =&\E[\log W_{e'}^{(l)}+\log W_{e''}^{(l)}
     +\E[\log((2\beta_v^{(l)\neq})^{-1})|W^{(l)}]]. 
       \label{eq:exp-log-induction-step}
  \end{align}
  By item \ref{item3} in Lemma
  \ref{le:aux-W-l-1}, given $W^{(l)}$, the random variable
  $(2\beta_v^{(l)\neq})^{-1}$ has an inverse Gaussian distribution
  $\ig( (W^{(l)}_{e'}+W^{(l)}_{e''})^{-1},1)$, hence
  its conditional expectation equals the first parameter
  $(W^{(l)}_{e'}+W^{(l)}_{e''})^{-1}$.
  In combination with Lemma \ref{le:bound-exp-X-alpha}
  and \eqref{eq:upper-bound-log-min} in the appendix,
  we obtain
  \begin{align}
    &\E[(2\beta_v^{(l)\neq})^{-\alpha}|W^{(l)}]\le
      (W^{(l)}_{e'}+W^{(l)}_{e''})^{-\alpha}\quad\text{for }\alpha\in[0,1] , \\
    &\E[(2\beta_v^{(l)\neq})^{-\alpha}|W^{(l)}]\le
      C_\alpha\quad\text{for }\alpha\in[0,\tfrac12) , \\
    &\E[\log((2\beta_v^{(l)\neq})^{-1})|W^{(l)}]
    \le\min\{-\log(W^{(l)}_{e'}+W^{(l)}_{e''}),\cnull\}.
  \end{align}
  The next step uses the inequality between arithmetic and geometric mean
  of two numbers $x,y>0$ in the following form
  \begin{align}
    \frac{xy}{x+y}=\frac12\frac{\sqrt{xy}}{\frac12(x+y)}\sqrt{xy}
    \le\frac12\sqrt{xy}. 
  \end{align}
  Next, we insert these inequalities in \eqref{eq:exp-W-hoch-alpha-induction-step}
  and \eqref{eq:exp-log-induction-step}.  
  Using Cauchy Schwarz in \eqref{eq:first-bound-alpha} and independence
  and identical distribution of
  $W_{e'}^{(l)}$ and $W_{e''}^{(l)}$ in \eqref{eq:second-bound-alpha}, we conclude 
  \begin{align}
    \E[(W^{(l-1)}_{\bar e})^\alpha]
    \le &\E\left[\left(\frac{W_{e'}^{(l)}W_{e''}^{(l)}}{W^{(l)}_{e'}+W^{(l)}_{e''}}\right)^\alpha
          \right]
          \le \E\left[\left(\frac12\sqrt{W_{e'}^{(l)}W_{e''}^{(l)}}\right)^\alpha
          \right]\nonumber\\
    \le &2^{-\alpha}\E[(W_{e'}^{(l)})^\alpha]^{\frac12}
                \E[(W_{e''}^{(l)})^\alpha]^{\frac12}=2^{-\alpha}\E[(W_{e'}^{(l)})^\alpha]
          \quad\text{for }\alpha\in[0,1],
    \label{eq:first-bound-alpha}\\
     \E[(W^{(l-1)}_{\bar e})^\alpha]
    \le &C_\alpha \E[(W_{e'}^{(l)}W_{e''}^{(l)})^\alpha]
          \nonumber\\
          \le &C_\alpha \E[(W_{e'}^{(l)})^\alpha]
                \E[(W_{e''}^{(l)})^\alpha]=C_\alpha \E[(W_{e'}^{(l)})^\alpha]^2
                \quad\text{for }\alpha\in[0,\tfrac12),
    \label{eq:second-bound-alpha}\\
        \E[\log W^{(l-1)}_{\bar e}]
    \le &\E\left[\log \frac{W_{e'}^{(l)}W_{e''}^{(l)}}{W^{(l)}_{e'}+W^{(l)}_{e''}}\right]
          \le \frac12(\E[\log W_{e'}^{(l)}]+\E[\log W_{e''}^{(l)}])-\log 2\nonumber\\
    = & \E[\log W_{e'}^{(l)}]-\log 2,\\
    \E[\log W^{(l-1)}_{\bar e}]
    \le & \E[\log W_{e'}^{(l)}]+\E[\log W_{e''}^{(l)}] +\cnull
          =2\E[\log W_{e'}^{(l)}] +\cnull.
  \end{align}
\end{proof}

One could alternatively obtain $\E[\log W^{(l-1)}_{\bar e}]\le \E[\log W^{(l)}_{e'}]-\log 2$ from
\eqref{eq:bound-exp-w-l-alpha} using $\lim_{\alpha\downarrow 0}(x^\alpha-1)/\alpha=\log x$ and
interchanging the limit $\alpha\downarrow 0$ with the expectation.

The following proof is based on iteration of the bounds in
Lemma \ref{le:bounds-W-l-one-step}. 

\medskip\noindent
\begin{proof}[Proof of Theorem \ref{thm:from-r-to-l}]
  Iterating \eqref{eq:bound-exp-w-l-alpha} $r-l$ times, we obtain
  \eqref{eq:phase1-alpha}. For $\alpha\in[0,\frac12)$,   
  the idea of the proof is to iterate \eqref{eq:bound-exp-w-l-alpha},
  starting with $l=r$, as long as it gives a better bound
  than \eqref{eq:bound2-exp-w-l-alpha}, and then switching over to
  the other bound \eqref{eq:bound2-exp-w-l-alpha}. 
  Let $m\in\{l,\ldots,r\}$.
  Iterating \eqref{eq:bound2-exp-w-l-alpha} $m-l$ times yields
  \begin{align}
     \label{eq:phase2-alpha}
   C_\alpha \E[(W^{(l)}_{\bar e})^\alpha]
    &\le (C_\alpha \E[(W^{(m)}_{\tilde e})^\alpha])^{2^{m-l}} 
  \end{align}
  for $\tilde e\in E_m$.
  Applying \eqref{eq:phase1-alpha} with $l$ replaced by $m$,
  i.e., $\E[(W^{(m)}_{\tilde e})^\alpha]\le (2^{-\alpha})^{r-m}\E[(W^{(r)}_{e'})^\alpha]$
  gives 
  \begin{align}
    C_\alpha \E[(W^{(l)}_{\bar e})^\alpha]
    \le \left(C_\alpha 
       (2^{-\alpha})^{r-m}\E[(W^{(r)}_{e'})^\alpha]\right)^{2^{m-l}}
  \end{align}
  and \eqref{eq:combined-alpha} follows using $W^{(r)}_{e'}=W_{e'}$. 
  To identify a minimizer, let
  $X_m:=C_\alpha (2^{-\alpha})^{r-m}\E[W_{e'}^\alpha]$.
  Note that the argument of the minimum in \eqref{eq:combined-alpha}
  is given by $X_m^{2^{m-l}}$. We
  observe the following equivalences
\begin{align}
   & X_m^{2^{m-l}}<X_{m-1}^{2^{m-1-l}}
    \;\Leftrightarrow\;
  (X_m^2)^{2^{m-1-l}}<(2^{-\alpha} X_m)^{2^{m-1-l}}
    \;\Leftrightarrow\;
     X_m<2^{-\alpha}
     \;\Leftrightarrow\nonumber\\
&  m<r-1-\alpha^{-1}\log_2(C_\alpha \E[W_{e'}^\alpha])
                               \;\Leftrightarrow\;
   m\le r-2-\lfloor\alpha^{-1}\log_2(C_\alpha \E[W_{e'}^\alpha])\rfloor
\end{align}
The claim about the minimizer $m_0$ in \eqref{eq:combined-alpha} follows. 

Let $0\le l\le m\le r$. Iterating $r-m$ times the inequality 
$\E[\log W^{(l-1)}_{\bar e}]\le \E[\log W^{(l)}_{e'}]-\log 2$
from \eqref{eq:bound2-exp-w-l-log} yields 
\begin{align}
      \label{eq:phase1-log}
   \E[\log W^{(m)}_{\tilde e}]&\le \E[\log W^{(r)}_{e'}]-(r-m)\log 2.
\end{align}
Iterating $m-l$ times the estimate 
$\E[\log W^{(l'-1)}_{\bar e}]+\cnull\le 2(\E[\log W^{(l')}_{e'}]+\cnull)$
from \eqref{eq:bound2-exp-w-l-log}, we obtain 
  \begin{align}
      \label{eq:phase2-log}
    \E[\log W^{(l)}_{\bar e}]+\cnull
    &\le 2^{m-l}(\E[\log W^{(m)}_{\tilde e}]+\cnull).
  \end{align}
  Inserting \eqref{eq:phase1-log}, the claim \eqref{eq:combined-log} follows.
  To identify a minimizer, set 
  $Y_m:=\E[\log W_{e'}]-(r-m)\log 2+\cnull=Y_{m-1}+\log 2$. Then, the
  argument of the minimum in \eqref{eq:combined-log} equals  $2^{m-l}Y_m$.
  The following are equivalent. 
  \begin{align}
   & 2^{m-l}Y_m < 2^{m-1-l}Y_{m-1}
    \;\Leftrightarrow\;
    Y_m<-\log 2
    \;\Leftrightarrow\;\\
    &m<r-1-(\log 2)^{-1}(\E[\log W^{(r)}_{e'}]+\cnull)
    \;\Leftrightarrow\;
      m\le r-2-\lfloor(\log 2)^{-1}(\E[\log W^{(r)}_{e'}]+\cnull)\rfloor
      \nonumber
  \end{align}
  The claim for the minimizer $m_1$ in \eqref{eq:combined-log} follows. 
 \end{proof}

 \subsection{Application to vrjp and errw}
 \label{subsec:application-to-vrjp-and-errw}

 In this section, we apply the recursive construction of effective random
 weights from the last section to restrictions of vrjp and errw on subdivided
 graphs.
In the following, we use the terms ``discrete-time process
associated with vrjp'' and its abbreviation ``discrete vrjp'' as synonyms.

 Since errw is a mixture of discrete vrjps with Gamma distributed
 i.i.d.\ weights, it makes sense not to start only with deterministic
 weights but more generally with i.i.d.\ weights, deterministic weights being a special case. 

 The next proof deals with an approximation of a possibly infinite graph by an
 increasing sequence of finite subgraphs. 
 
 \medskip\noindent
\begin{proof}[Proof of Theorem \ref{thm:vrjp-on-G-r-pos-rec}]
  Consider an increasing sequence $\Lambda^N\uparrow\Lambda$, $N\in\N_0$, of finite connected
  vertex sets in $G$ with $\rho\in\Lambda_0$. Let $G^N$ be the subgraph of $G$ with
  vertex set $\Lambda^N$ and edge set $E^N=\{e\in E:\; e\subseteq\Lambda^N\}$, and
  $G^N_r=(\Lambda^N_r,E^N_r)$ be the corresponding subdivided graph, where every edge in $E^N$ has been
  replaced by a series of $2^r$ edges. Consider the discrete vrjp $X^N$ on $G^N_r$ with
  random weights $W_e$, $e\in E^N_r$.  By Theorem \ref{thm:mixture-of-vrjps}, its
  restriction $(X^N)^{\Lambda_l^N\neq}$ to $G^N_l$ is a mixture of discrete vrjps 
  with random weights $W_{\bar e}^{(l)}$, $\bar e\in E_l^N$, fulfilling the
  properties in Lemma \ref{le:aux-W-l-1}. We emphasize that the finite-dimensional
  marginals $(W_{\bar e}^{(l)})_{\bar e\in F}$, $F\subseteq E_l$ finite, 
  do not depend on the size $N$ of the graph, whenever $N$ is large enough so
  that $F\subseteq E_l^N$. This allows us to take
  the same random variables $W_{\bar e}^{(l)}$ for all $N$. Moreover, by
  part~\ref{item5} of Lemma \ref{le:aux-W-l-1}, the family of weights $(W_e^{(l)})_{e\in E_l}$
  is independent or i.i.d., respectively, if $(W_e)_{e\in E}$ has this
  property. 

  Given $m\in\N$, consider only the
  first $m$ steps of the restrictions $X^{\Lambda_l\neq}$ and $(X^N)^{\Lambda_l^N\neq}$. 
  If $N$ is large enough, these two restrictions have the same law because they have
  no chance to enter $\Lambda_l\setminus\Lambda_l^N$. It follows that $X^{\Lambda_l\neq}$
  is a mixture of discrete vrjp with random weights $W_{\bar e}^{(l)}$, $\bar e\in E_l$
  because this is true for its restriction to any given number $m$ of steps. 
    Using the estimate
  \eqref{eq:phase1-alpha} from Theorem \ref{thm:from-r-to-l}
  and the assumption on $\E[W_e^\alpha]$ for some
  $\alpha\in(0,\frac14]$, we obtain for all $\bar e\in E_l$ and $e\in E_r$, 
  \begin{align}
    \E[(W^{(l)}_{\bar e})^\alpha]
    &\le (2^{-\alpha})^{r-l}\E[(W_e)^\alpha]\le\cdrei.
  \end{align}
  The claim follows from Fact \ref{fact:sabot-tarres-recurrence}. 
\end{proof}

Finally, the result for errw is obtained by specializing the i.i.d.\
weights to i.i.d.\ Gamma distributed weights as follows. 

\medskip\noindent
\begin{proof}[Proof of Theorem \ref{thm:Errw-on-subdivided-graphs}]
  By \cite[Theorem 1]{sabot-tarres2012}, the edge-reinforced random walk on $G_r$ is a
  mixture of the discrete vrjp with i.i.d.\ Gamma($a$,1)-distributed
  weights $W_e$, $e\in E_r$. We observe that for all $\alpha\ge 0$
  one has $\E[W_e^\alpha]=\Gamma(a+\alpha)/\Gamma(a)$. Consequently, the claim
  follows from Theorem \ref{thm:vrjp-on-G-r-pos-rec}.
\end{proof}

\section{Proofs for the non-linear hyperbolic sigma model}
\label{sec:susy-proofs}

Recall the setup of Section \ref{subsec:Non-linear-hyperbolic-supersymmetric-sigma-model}. 
Set $\hthreeplus:=\{\sigma=(x,y,z,\xi,\eta)\in\cA_0^3\times\cA_1^2:\sk{\sigma,\sigma}<0,
\body z>0\}$. 
For $\sigma\in\hthreeplus$, let $\|\sigma\|:=\sqrt{-\sk{\sigma,\sigma}}$. 
The following lemma describes the super-Laplace transform of the canonical
superintegration form $\cD\sigma$ on $\htwo$. 
Recall from Section \ref{subsec:Non-linear-hyperbolic-supersymmetric-sigma-model}
that the graph $(\Lambda,E_+)$ with edge set
$E_+=\{\{i,j\}\subseteq\Lambda:W_{ij}>0\}$ is connected, which implies
that at least one $W_{1i}$ is strictly positive for every vertex $1\in\Lambda$.

\begin{lemma}[Integration of one variable]
  \label{le:integrate-one-spin}
  Let $1\in\Lambda$ and set $1^c=\Lambda\setminus\{1\}$.
  One has
  \begin{align}
    \label{eq:integrate-one-var-isolated}
    \int_{\htwo}\cD\sigma_1\, e^{\sum_{i\in 1^c}W_{1i}\sk{\sigma_1,\sigma_i}}
    =e^{-\|\sum_{i\in 1^c}W_{1i}\sigma_i\|},
  \end{align}
  and consequently, 
\begin{align}
  \label{eq:integrate-one-var}
  \int_{\htwo}\cD\sigma_1\,
    e^{\frac12\sum_{i,j\in\Lambda}W_{ij}(1+\sk{\sigma_i,\sigma_j})}
      =& e^{\frac12\sum_{i,j\in 1^c}W_{ij}(1+\sk{\sigma_i,\sigma_j})
    +\sum_{i\in 1^c}W_{1i}-\|\sum_{i\in 1^c}W_{1i}\sigma_i\|}.
\end{align}
\end{lemma}
\begin{proof}
  By the convexity of $\hthree_+$, one has
    $\sum_{i\in 1^c}W_{1i}\sigma_i\in\hthree_+$ and hence 
    \cite[Lemma 3.2, second equality in (3.1)]{disertori-merkl-rolles2020}
    is applicable and yields 
  \begin{align}
    \int_{\htwo}\cD\sigma_1\, e^{\sum_{i\in 1^c}W_{1i}\sk{\sigma_1,\sigma_i}}
    =& \int_{\htwo}\cD\sigma_1\, e^{\sk{\sigma_1,\sum_{i\in 1^c}W_{1i}\sigma_i}}
       =e^{-\|\sum_{i\in 1^c}W_{1i}\sigma_i\|}.
  \end{align}
  The second claim \eqref{eq:integrate-one-var} follows from
  \eqref{eq:integrate-one-var-isolated} using $\sk{\sigma_1,\sigma_1}=-1$
  and decomposing the exponent as follows 
  \begin{align}
    \frac12\sum_{i,j\in\Lambda}W_{ij}(1+\sk{\sigma_i,\sigma_j})
    =\frac12\sum_{i,j\in 1^c}W_{ij}(1+\sk{\sigma_i,\sigma_j})
    +\sum_{i\in 1^c}W_{1i}+\sum_{i\in 1^c}W_{1i}\sk{\sigma_1,\sigma_i}.
  \end{align}
\end{proof}

This lemma is the key ingredient for analyzing the restriction of
the $\htwo$ model. 

\medskip\noindent
\begin{proof}[Proof of Theorem \ref{thm:mixture-of-htwo}]
    The second equation in \eqref{eq:mixture-of-h22} follows from the
    restriction property of the betas; see Remark
    \ref{rem:restriction-conditioning-property-beta}. 
    The proof of the first equation in \eqref{eq:mixture-of-h22} is by
    induction with respect to the
  cardinality of $I$. For $I=\emptyset$, there is nothing to prove.
  As induction hypothesis, assume that formula \eqref{eq:mixture-of-h22} holds
  for given $I\subsetneq\Lambda\setminus\{\rho\}$ and $J=\Lambda\setminus I$. For the
  induction step, take $i\in J$, $i\neq\rho$, 
  and set $\tilde I=I\cup\{i\}$, $\tilde J=\Lambda\setminus\tilde I=J\setminus\{i\}$.
  For any superfunction $f$ on $(\htwo)^{\tilde J}$ being compactly supported
  or at least sufficiently fast
  decaying so that the integral on the left-hand side of \eqref{eq:first-step-induction} is
  well-defined, the induction hypothesis yields
  \begin{align}
    \label{eq:first-step-induction}
    \int_{(\htwo)^\Lambda}\mu_\Lambda^W(\sigma_\Lambda)f(\sigma_{\tilde J})=
    \int_{\R^\Lambda}\nu_\Lambda^W(d\beta)\int_{(\htwo)^J}\mu_J^{W^J(\beta_I)}(\sigma_J)f(\sigma_{\tilde J}).
  \end{align}
  We fix $\beta\in\R^\Lambda$, abbreviate $W^J=W^J(\beta_I)$ when there
  is no risk of confusion, and set
  $g(\sigma_{\tilde J})=e^{\frac12\sum_{j,k\in\tilde J}W^J_{jk}(1+\sk{\sigma_j,\sigma_k})}
  f(\sigma_{\tilde J})$.
  We split the integrand into a part which does not involve
  $\sigma_i$ and the remaining part involving $\sk{\sigma_i,\sigma_j}$,
  $j\in\tilde J$: 
  \begin{align}
    \label{eq:split}
&    \int_{(\htwo)^J}\mu_J^{W^J(\beta_I)}(\sigma_J)f(\sigma_{\tilde J})
    =\int_{(\htwo)^J}\cD\sigma_J\, e^{\frac12\sum_{j,k\in J}W^J_{ij}(1+\sk{\sigma_j,\sigma_k})}
                       f(\sigma_{\tilde J})\nonumber\\
    &=\int_{(\htwo)^{\tilde J}}\cD\sigma_{\tilde J}\,g(\sigma_{\tilde J})
    \int_{\htwo}\cD\sigma_i
      e^{\sum_{j\in\tilde J}W^J_{ij}(1+\sk{\sigma_i,\sigma_j})}.
  \end{align}
  Note that the term $W^J_{ii}(1+\sk{\sigma_i,\sigma_i})=0$
  has been dropped. 
  By formula~\eqref{eq:integrate-one-var-isolated} in Lemma~\ref{le:integrate-one-spin}
  and using the abbreviation 
  $W^J_i:=\sum_{j\in \tilde J}W^J_{ij}
  =\sum_{j\in J\setminus\{i\}}W^J_{ij}$, the single-spin integral in the last
  expression equals
  \begin{align}
    &\int_{\htwo}\cD\sigma_i
    e^{\sum_{j\in\tilde J}W^J_{ij}(1+\sk{\sigma_i,\sigma_j})}
      =e^{W^J_i-\left\|\sum_{j\in \tilde{J}}W^J_{ij}\sigma_j\right\|}
      =e^{W^J_i(1-\|\sum_{j\in \tilde{J}}
    \frac{W^J_{ij}}{W^J_i}\sigma_j\|)}.
    \label{eq:Integrate-single-spin}
  \end{align}
  Using auxiliary random variables 
  $X\sim\ig\left(\frac{W^J_i}{2},\frac{(W^J_i)^2}{2}\right)$
  and $Y=\frac{2X}{(W^J_i)^2}\sim
  \ig\left((W^J_i)^{-1},1\right)$
  with inverse Gaussian distributions, cf.\ Appendix \ref{sec:facts-IG}, 
  on some auxiliary probability space with expectation operator
  denoted by $E$, Lemma~\ref{le:laplace-IG} allows us to rewrite
  the last expression in the form
  \begin{align}
    \eqref{eq:Integrate-single-spin}
    &=E\left[ e^{X(1-\|\sum_{j\in \tilde{J}}\frac{W^J_{ij}}{W^J_i}\sigma_j\|^2)}\right]
      =E\left[ e^{
      \frac{X}{(W^J_i)^2}
      \left(
      (W^J_i)^2+\sum_{j,k\in\tilde{J}}W^J_{ij}W^J_{ik}
      \sk{\sigma_j,\sigma_k}
      \right)}
    \right]
      \nonumber\\&=
      E\left[e^{\frac{Y}{2}
      \sum_{j,k\in\tilde{J}}W^J_{ij}W^J_{ik}
      \left(1+\sk{\sigma_j,\sigma_k}\right)}\right]. 
  \end{align}
  Using the notation \eqref{eq:def-H-beta}, we take the specific choice $(\R^J,\mathcal{B}(\R^J),
  \nu_J^{W^J}(d\tilde \beta))$ and
  $Y=(2\tilde{\beta_i}-W^J_{ii})^{-1}=([H_{J,\tilde\beta}^{W^J}]_{ii})^{-1}$ for the auxiliary probability space and the auxiliary random variable $Y$, respectively.
  Indeed, $Y\sim\ig((W^J_i)^{-1},1)$ for this choice
  is a consequence \eqref{eq:marginal-beta-IG}.
  Summarizing, this shows
  \begin{align}
    \int_{\htwo}\cD\sigma_i
      e^{\sum_{j\in\tilde J}W^J_{ij}(1+\sk{\sigma_i,\sigma_j})}
    =&
      \int_{\R^J} \nu_J^{W^J}(d\tilde \beta)
      e^{\frac{1}{2}
      \sum_{j,k\in\tilde{J}}\frac{W^J_{ij}W^J_{ik}}{2\tilde{\beta_i}-W^J_{ii}}
      \left(1+\sk{\sigma_j,\sigma_k}\right)}.
  \end{align}
  We substitute this in
  \eqref{eq:split} to obtain
  \begin{align}
    &\int_{(\htwo)^{\tilde J}}\mu_J^{W^J(\beta_I)}(\sigma_J)f(\sigma_{\tilde J})
    =\int_{(\htwo)^{\tilde J}}\cD\sigma_{\tilde J}\,g(\sigma_{\tilde J})
      \int_{\R^J} \nu_J^{W^J}(d\tilde \beta)
      e^{\frac{1}{2}
      \sum_{j,k\in\tilde{J}}\frac{W^J_{ij}W^J_{ik}}{2\tilde{\beta_i}-W^J_{ii}}
    \left(1+\sk{\sigma_j,\sigma_k}\right)}
      \nonumber
    \\&=
    \int_{\R^J} \nu_J^{W^J}(d\tilde \beta)
    \int_{(\htwo)^{\tilde J}}\cD\sigma_{\tilde J}\,f(\sigma_{\tilde J})
    e^{\frac{1}{2}
    \sum_{j,k\in\tilde{J}}
    \left[W^J_{jk}+\frac{W^J_{ji}W^J_{ik}}{2\tilde{\beta_i}-W^J_{ii}}\right]
    \left(1+\sk{\sigma_j,\sigma_k}\right)}\nonumber\\
    &=\int_{\R^J} \nu_J^{W^J}(d\tilde \beta)
      \int_{(\htwo)^{\tilde J}}\mu_{\tilde J}^{W^{\tilde J}(\beta_I,\tilde\beta_i)}
      (\sigma_{\tilde J})\, f(\sigma_{\tilde J})
  \end{align}
because 
  \begin{align}
    \left(W^J_{jk}+
    \frac{W^J_{ji}W^J_{ik}}{2\tilde{\beta_i}-W^J_{ii}}
    \right)_{j,k\in\tilde J}
    =&W^J_{\tilde{J}\tilde{J}}
    +W^J_{\tilde{J}i}\left[\left(H_{J,\tilde\beta}^{W^J}\right)_{ii}\right]^{-1}
    W^J_{i\tilde{J}}\nonumber\\
    =&(W^J)^{\tilde{J}}(\tilde\beta_i)=W^{\tilde J}(\beta_I,\tilde\beta_i), 
  \end{align}
   where we used $W^J=W^J(\beta_I)$ and \eqref{eq:W-tilde-J-W-J} with $J\setminus\tilde J=\{i\}$. 
  Taking now $\beta$ random, we insert the result in
  \eqref{eq:first-step-induction} and obtain
  \begin{align}
    \int_{(\htwo)^\Lambda}\mu_\Lambda^W(\sigma_\Lambda)f(\sigma_{\tilde J})=
    \int_{\R^\Lambda}\nu_\Lambda^W(d\beta)
\int_{\R^J} \nu_J^{W^J(\beta_I)}(d\tilde \beta)
      \int_{(\htwo)^{\tilde J}}\mu_{\tilde J}^{W^{\tilde J}(\beta_I,\tilde\beta_i)}
    (\sigma_{\tilde J})\, f(\sigma_{\tilde J}).
  \end{align}
  By the conditioning property cited in Remark \ref{rem:restriction-conditioning-property-beta}, the conditional law of $\beta_J$ 
  given $\beta_I$ with respect to $\nu_\Lambda^W$ equals $\nu_J^{W^J(\beta_I)}$.
  In other words, for any integrable function $h(\beta_I,\beta_J)$,
  one has 
  \begin{align}
    \int_{\R^\Lambda}\nu_\Lambda^W(d\beta)
    \int_{\R^J}\nu_J^{W^J(\beta_I)}(d\tilde\beta)\, h(\beta_I,\tilde\beta)
    =  \int_{\R^\Lambda}\nu_\Lambda^W(d\beta)\, h(\beta_I,\beta_J).
  \end{align}
  Recall $i\in J$. Applying the last identity with
  \begin{align}
    h(\beta_I,\beta_J)=\int_{(\htwo)^{\tilde J}}\mu_{\tilde J}^{W^{\tilde J}(\beta_I,(\beta_J)_i)}
    (\sigma_{\tilde J})\, f(\sigma_{\tilde J}),
  \end{align}
  and observing $(\beta_I,(\beta_J)_i)=\beta_{\tilde I}$, 
  we conclude the induction step with the calculation
  \begin{align}
    &\int_{(\htwo)^\Lambda}\mu_\Lambda^W(\sigma_\Lambda)f(\sigma_{\tilde J})
      =\int_{\R^\Lambda}\nu_\Lambda^W(d\beta)\int_{(\htwo)^{\tilde J}}\mu_{\tilde J}^{W^{\tilde J}(\beta_{\tilde I})}(\sigma_{\tilde J})f(\sigma_{\tilde J}).
  \end{align}
\end{proof}

 \begin{appendix}
   \section{Facts about inverse Gaussians}
   \label{sec:facts-IG}
The density of an inverse Gaussian distribution $\ig(\mu,\lambda)$
with parameters $\mu,\lambda>0$ is given by
\begin{align}
  \label{eq:density-IG}
  f(x)=\sqrt{\frac{\lambda}{2\pi x^3}}
  \exp\left(-\frac{\lambda(x-\mu)^2}{2\mu^2x}\right), \quad x>0. 
\end{align}
For $X\sim \ig(\mu,\lambda)$, one has $E[X]=\mu$ and the Laplace transform is given by
\begin{align}
  \label{eq:laplace-trafo-IG}
  \phi(t)=E[e^{tX}]
  =e^{\frac{\lambda}{\mu}\left(1-\sqrt{1-\frac{2\mu^2t}{\lambda}}\right)},
  \quad t\le\frac{\lambda}{2\mu^2}. 
\end{align}
The following special case of the Laplace transform is used in the
proof of Theorem \ref{thm:mixture-of-htwo}. 

\begin{lemma}[Laplace transform]
  \label{le:laplace-IG}
  If $X\sim \ig(\frac{a}{2},\frac{a^2}{2})$ for some $a>0$, then
  \begin{align}
    E[e^{(1-s)X}]=e^{a(1-\sqrt s)}, \quad s\ge 0. 
  \end{align}
\end{lemma}
\begin{proof}
  Setting $\mu=\frac{a}{2}$ and $\lambda=\frac{a^2}{2}$, this follows from
  \eqref{eq:laplace-trafo-IG} using $\frac{\lambda}{\mu}=a$ and 
  $\frac{2\mu^2}{\lambda}=1$. 
\end{proof}

Our estimates for the effective weights in Theorem \ref{thm:from-r-to-l} require the
estimates in the following Lemmas~\ref{le:bound-exp-X-alpha} and \ref{le:estimate-log-IG}. 

\begin{lemma}[Estimates for some moments]
  \label{le:bound-exp-X-alpha}
  For $W>0$ and $X_W\sim \ig(W^{-1},1)$, with the constant $C_\alpha$ from Theorem
  \ref{thm:from-r-to-l}, one has
  \begin{align}
    \label{eq:bound-exp-X-hoch-alpha}
    E[X_W^\alpha]\le & W^{-\alpha}\hspace{2.3cm}\text{ for }\alpha\in[0,1], \\
    E[X_W^\alpha]\le &  C_\alpha
                       \quad\text{ for }\alpha\in[0,\tfrac12), \qquad
                       \lim_{W\downarrow 0}E[X_W^\alpha]=C_\alpha. 
                       \label{eq:bound-exp-X-hoch-alpha-C-alpha}
  \end{align}
\end{lemma}
\begin{proof}
  Using Jensen's inequality for the concave function $x^\alpha$,
  we obtain $E[X_W^\alpha]\le E[X_W]^\alpha=W^{-\alpha}$ for $0\le \alpha\le 1$. 

  We prove now the second bound. 
  By \eqref{eq:density-IG}, one has 
  \begin{align}
    \label{eq:exp-x-hoch-alpha-with-density}
    f_\alpha(W):=E[X_W^\alpha]=\frac{1}{\sqrt{2\pi}}\int_0^\infty
    x^{\alpha-\frac32}e^{-\frac{(Wx-1)^2}{2x}}\, dx
  \end{align}
  for all $\alpha\in[0,\frac12)$. 
  Taking the derivative with respect to $W$ and
  substituting first $y=Wx$ and then $z=1/y$ it follows
\begin{align}
    f_\alpha'(W)=&\frac{1}{\sqrt{2\pi}}\int_0^\infty
    x^{\alpha-\frac32}(1-Wx)e^{-\frac{W}{2}(Wx+(Wx)^{-1}-2)}
                   \, dx
  \nonumber\\
  =&\frac{W^{\frac12-\alpha}}{\sqrt{2\pi}}\int_0^\infty
    y^{\alpha-\frac32}(1-y)
     e^{-\frac{W}{2}(y+y^{-1}-2)}\, dy
     \nonumber\\
=&-\frac{W^{\frac12-\alpha}}{\sqrt{2\pi}}\int_0^\infty
    z^{-\alpha-\frac32}(1-z)
   e^{-\frac{W}{2}(z+z^{-1}-2)}\, dz
  .
\end{align}
Taking the average of the last two expressions and using
$(y^\alpha-y^{-\alpha})(1-y)\le 0$ for $y\ge 0$ and $\alpha>0$ yields 
\begin{align}
  f_\alpha'(W)
  =&\frac{W^{\frac12-\alpha}}{2\sqrt{2\pi}}\int_0^\infty
    y^{-\frac32}(y^\alpha-y^{-\alpha})(1-y)
    e^{-\frac{W}{2}(y+y^{-1}-2)}\, dy
     \le 0 .
\end{align}
Hence, the function $f_\alpha$ is decreasing, which implies
$f_\alpha(W)\le \lim_{w\downarrow 0}f_\alpha(w)$ for all $W>0$.
For $0\le \alpha<\frac12$,
we evaluate this limit as follows. Using
dominated convergence 
with the integrable upper bound $x^{\alpha-\frac32} e^{1-\frac{1}{2x}}$
for the integrand in \eqref{eq:exp-x-hoch-alpha-with-density} with
$0<W\le 1$ and then substituting $y=(2x)^{-1}$ , we conclude
\begin{align}
f_\alpha(W)\le \lim_{w\downarrow 0}f_\alpha(w)
=\frac{1}{\sqrt{2\pi}}\int_0^\infty
  x^{\alpha-\frac32} e^{-\frac{1}{2x}}\, dx
=\frac{1}{2^\alpha\sqrt{\pi}}\int_0^\infty
  y^{-\alpha-\frac12} e^{-y}\, dy
  =C_\alpha.
\end{align}
\end{proof}

\begin{lemma}[Estimate for the logarithmic moment]
  \label{le:estimate-log-IG}
  \mbox{}\\
  For $W>0$ and $X_W\sim \ig(W^{-1},1)$, one has 
  \begin{align}
    \label{eq:exp-log-X}
    E[\log X_W]=&-\log W-\int_0^\infty\frac{e^{-u}}{u+2W}\, du
                  =e^{2W}\int_0^{2W}(\log t+\gamma)e^{-t}\, dt+\cnull
  \end{align}
  with $\cnull$ and $\gamma$ specified in Theorem \ref{thm:from-r-to-l}.
  Furthermore, the following bounds hold
  \begin{align}
    \label{eq:upper-bound-log-min}
    -\log(W+\tfrac12)\le E[\log X_W]&\le\min\{-\log W,\cnull\}\quad\text{for }W>0,\\
    \label{eq:bound-log-X-small-W}
    \cnull+4W(\log W+\cnull-1) \le E[\log X_W]&\le\cnull\quad\text{for }0<W\le\tfrac12 e^{-\gamma},\\
    \label{eq:exp-log-X-lim-W-to-0}
    \lim_{W\to 0}E[\log X_W]&=\cnull.
  \end{align}
\end{lemma}
\begin{proof}
  By \eqref{eq:density-IG}, the density of $X_W$ is given by 
  \begin{align}
    f(x)=\sqrt{\frac{1}{2\pi x^3}} e^{-\frac{(Wx-1)^2}{2x}}, 
  \end{align}
  $x>0$. Hence, $\log X_W$ has the density
  \begin{align}
    g(u)=f(e^u)e^u
    =\frac{1}{\sqrt{2\pi}}e^{-\frac12(We^u-1)^2e^{-u}}e^{-\frac{u}{2}},
  \end{align}
  $u\in\R$, with respect to the Lebesgue measure on $\R$.
  Rewriting the exponent of the first exponential as 
  \begin{align}
    \frac12(We^u-1)^2e^{-u}
    =\frac12W(We^u+(We^u)^{-1}-2)=W[\cosh(u+\log W)-1]
  \end{align}
  yields the density
 \begin{align}
    \label{eq:density-log-1-durch-beta}
   g(u)=\frac{1}{\sqrt{2\pi}}
    e^{-W[\cosh(u+\log W)-1]-\frac{u}{2}}, \quad u\in\R. 
  \end{align}
  Using the substitution $v=u+\log W$, we obtain
  \begin{align}
    E[\log X_W]=& \frac{1}{\sqrt{2\pi}}\int_\R u
    e^{-W[\cosh(u+\log W)-1]-\frac{u}{2}}\, du\nonumber\\
    =&-\log W+\sqrt{\frac{W}{2\pi}}e^W\int_\R v
       e^{-W\cosh v-\frac{v}{2}}\, dv.
       \label{eq:intermediate-log-X}
  \end{align}
Recall the modified Bessel function of second kind $K_p$, $p\in\R$. 
By \cite[9.6.24]{abramowitz-stegun}, for $w>0$, one has
\begin{align}
  K_p(w)=\int_0^\infty e^{-w\cosh v}\cosh(p v)\, dv
  =\frac12\int_\R e^{-w\cosh v}e^{p v}\, dv. 
\end{align}
Taking the derivative with respect to $p$ yields
\begin{align}
  \label{eq:partial-p-K-p}
  \partial_pK_p(w)
  =\frac12\int_\R v e^{-w\cosh v} e^{p v}\, dv.
\end{align}
Combining this with \eqref{eq:intermediate-log-X}, we obtain 
\begin{align}
  E[\log X_W]=&-\log W+\sqrt{\frac{2W}{\pi}}e^W \partial_pK_p(W)|_{p=-\frac12}. 
\end{align}
By \cite[Page 1]{ryavec-2021-besselk},
\begin{align}
   \partial_pK_p(w)|_{p=\frac12}=&K_{\frac12}(w)\int_0^\infty\frac{e^{-u}}{u+2w}\, du.
\end{align}
Combining this with $K_p(w)=K_{-p}(w)$, cf.\ \cite[9.6.6]{abramowitz-stegun},
and $K_{\frac12}(w)=\sqrt{\pi/(2w)}e^{-w}$, cf.\ \cite[10.2.17]{abramowitz-stegun},
yields
\begin{align}
  \label{eq:partial-p-K-p-2}
  \partial_pK_p(w)|_{p=-\frac12}(w)=-\partial_pK_p(w)|_{p=\frac12}(w)
  =&-\sqrt{\frac{\pi}{2w}}e^{-w}\int_0^\infty\frac{e^{-u}}{u+2w}\, du.
\end{align}
This proves the first equality in the claim \eqref{eq:exp-log-X}. In particular,
$E[\log X_W]\le -\log W$.

To prove the second equality in \eqref{eq:exp-log-X}, we observe that
  \begin{align}
    \label{eq:int-log-t-plus-gamma-0}
    \int_0^\infty(\log t+\gamma)e^{-t}\, dt=0  
  \end{align}
   by the expression
  \eqref{eq:euler-mascheroni} for the Euler
  Mascheroni constant; note that the integrand is integrable both at
  $t=0$ and $t=\infty$. Thus, using integration by parts, we conclude
  \begin{align}
   & e^{2W}\int_0^{2W}(\log t+\gamma)e^{-t}\, dt
    =-e^{2W}\int_{2W}^\infty(\log t+\gamma)e^{-t}\, dt\nonumber\\
    =&-(\log(2W)+\gamma)-\int_{2W}^\infty\frac{e^{2W-t}}{t}\, dt
       =-\log W-\log 2-\gamma-\int_0^\infty\frac{e^{-u}}{u+2W}\, du.
  \end{align}
  Next, we show $f(W):=\int_0^{2W}(\log t+\gamma)e^{-t}\, dt\le 0$.
  Note that the integrand $(\log t+\gamma)e^{-t}$
  is negative for $0<t<e^{-\gamma}$ and positive for $t>e^{-\gamma}$.
  Hence, the function $f$ is decreasing
  on the interval $[0,e^{-\gamma}]$ and increasing on $[e^{-\gamma},\infty]$.
  Since $f(0)=0=f(\infty)$ by \eqref{eq:int-log-t-plus-gamma-0},
  the claim $f\le 0$ and hence $E[\log X_W]\le\cnull$ follow. We conclude that
  the upper bound in \eqref{eq:upper-bound-log-min} is valid.

  We proceed with the lower bound in \eqref{eq:upper-bound-log-min}.
  We have
  \begin{align}
    \int_0^\infty\frac{e^{-u}}{u+2W}\, du
    =e^{2W}E_1(2W)\le\log\left(1+\frac{1}{2W}\right)
  \end{align}
  with the exponential integral $E_1(x)=\int_x^\infty\frac{e^{-u}}{u}\, du$
  using the known bound $e^x E_1(x)\le\log(1+1/x)$ for all $x>0$, cf.\
  \cite[5.1.20]{abramowitz-stegun}. Substituting this in \eqref{eq:exp-log-X}, 
  we conclude
  \begin{align}
    E[\log X_W]\ge -\log W - \log\left(1+\frac{1}{2W}\right)
    =-\log(W+\tfrac12). 
  \end{align}
  Assume now $0<W\le\frac12 e^{-\gamma}$. 
  The upper bound in \eqref{eq:bound-log-X-small-W} is already contained
  in \eqref{eq:upper-bound-log-min}; it remains to show the lower bound.
  We estimate the last integral in \eqref{eq:exp-log-X}, using that
  $e^{2W}\le\exp(e^{-\gamma})=1.753\ldots\le 2$ and $\log t+\gamma\le 0$
  hold for $0<t\le e^{-\gamma}$, to obtain 
  \begin{align}
    E[\log X_W]-\cnull
    =&e^{2W}\int_0^{2W}(\log t+\gamma)e^{-t}\, dt
       \ge 2\int_0^{2W}(\log t+\gamma)\, dt\nonumber\\
    =&4W(\log W+\cnull-1).
  \end{align}
  This proves \eqref{eq:bound-log-X-small-W} and \eqref{eq:exp-log-X-lim-W-to-0}. 
\end{proof}

We remark that the upper bound $E[\log X_w]\le\cnull$ in \eqref{eq:bound-log-X-small-W}
can also be obtained from \eqref{eq:bound-exp-X-hoch-alpha-C-alpha}
in the form $E[(X_W^\alpha-1)/\alpha]\le (C_\alpha-1)/\alpha$, taking
the limit as $\alpha\downarrow 0$, interchanging limit and expectation,
and using $\partial_\alpha C_\alpha|_{\alpha=0}=\cnull$.
\end{appendix}

\paragraph{Acknowledgements.}
This work is funded by the Deutsche Forschungsgemeinschaft (DFG, German Research Foundation) within SPP 2265 Random Geometric Systems (Project number 443757604)
and the Hausdorff Center for Mathematics (project ID 390685813). 
The authors would like to thank the Isaac Newton Institute for Mathematical Sciences, Cambridge, for support and hospitality during
the programme \emph{Stochastic systems for anomalous diffusion} (EPSRC grant no EP/R014604/1.34), 
where part of this work was undertaken. 
 
\bibliographystyle{alpha}

\end{document}